\documentclass[11pt,a4paper]{article}
\RequirePackage{fix-cm}
\usepackage{amsfonts,amsmath,amsthm,epsfig}
\newtheorem{theorem}{Theorem}
\newtheorem{proposition}[theorem]{Proposition}
\newtheorem{remark}[theorem]{Remark}
\numberwithin{theorem}{section}
\newtheorem{lemma}[theorem]{Lemma}
\numberwithin{equation}{section}
\usepackage{graphicx}
\usepackage{pdfsync}

\usepackage[hidelinks]{hyperref}
\usepackage{latexsym,amsfonts,amsmath,graphics,multirow}
\usepackage{mathrsfs}

\usepackage[normalem]{ulem}
\usepackage[usenames,dvipsnames,svgnames,table]{xcolor}

\usepackage{epsfig}
\usepackage[font=scriptsize]{caption}
\usepackage{algorithm,algorithmic}
\usepackage{makecell}

\numberwithin{equation}{section}

\newcommand{\mask}[1]{}

\definecolor{darkred}{RGB}{139,0,0}
\definecolor{darkgreen}{RGB}{0,100,0}
\definecolor{darkmagenta}{RGB}{139,0,139}
\definecolor{orange}{RGB}{207,83,0}
\definecolor{brown}{RGB}{139,69,19}

\newcommand{\bsDelta}{{\boldsymbol{\Delta}}}
\newcommand{\bsgamma}{{\boldsymbol{\gamma}}}

\newcommand{\bsnu}{{\boldsymbol{\nu}}}

\newcommand{\bsrho}{{\boldsymbol{\rho}}}
\newcommand{\bszero}{{\boldsymbol{0}}}

\newcommand{\bsb}{{\boldsymbol{b}}}
\newcommand{\bse}{{\boldsymbol{e}}}

\newcommand{\bsm}{{\boldsymbol{m}}}

\newcommand{\bsx}{{\boldsymbol{x}}}
\newcommand{\bsy}{{\boldsymbol{y}}}

\newcommand{\rd}{\mathrm{d}}

\newcommand{\bbA}{\mathbb{A}}

\newcommand{\bbR}{\mathbb{R}}

\newcommand{\bbN}{\mathbb{N}}
\newcommand{\bbE}{\mathbb{E}}
\newcommand{\calO}{\mathcal{O}}
\newcommand{\calC}{\mathcal{C}}

\newcommand{\calI}{\mathcal{I}}

\newcommand{\calR}{\mathcal{R}}
\newcommand{\calW}{\mathcal{W}}
\newcommand{\calX}{\mathcal{X}}
\newcommand{\calY}{\mathcal{Y}}

\newcommand{\satop}[2]{\stackrel{\scriptstyle{#1}}{\scriptstyle{#2}}}

\newcommand{\setu}{\mathrm{\mathfrak{u}}}

\newcommand{\supp}{\mathrm{supp}}

\title{Parabolic PDE-constrained optimal control under uncertainty with entropic risk measure using \\quasi-Monte Carlo integration
}

\author{Philipp A. Guth\footnotemark[2] \and
Vesa Kaarnioja\footnotemark[3] \and
Frances Y. Kuo\footnotemark[4] \and 
Claudia Schillings\footnotemark[3] \and 
Ian H. Sloan\footnotemark[4]
}

\begin{document}
\maketitle
\footnotetext[2]{Johann Radon Institute for Computational and Applied Mathematics, \"OAW, Altenbergerstrasse 69, 4040 Linz, Austria. E-mail: \href{mailto:philipp.guth@ricam.oeaw.ac.at}{philipp.guth@ricam.oeaw.ac.at}}
\footnotetext[3]{Fachbereich Mathematik und Informatik, Freie Universit\"at Berlin, Arnimallee 6, 14195 Berlin, Germany. E-mail: \href{mailto:vesa.kaarnioja@fu-berlin.de}{vesa.kaarnioja@fu-berlin.de} $\cdot$ \href{mailto:c.schillings@fu-berlin.de}{c.schillings@fu-berlin.de}}
\footnotetext[4]{School of Mathematics and Statistics, UNSW Sydney, Sydney NSW 2052, Australia. E-mail: \href{mailto:f.kuo@unsw.edu.au}{f.kuo@unsw.edu.au} $\cdot$ \href{mailto:i.sloan@unsw.edu.au}{i.sloan@unsw.edu.au}}
\begin{abstract}
We study the application of a tailored quasi-Monte Carlo (QMC) method to a class
of optimal control problems subject to parabolic partial differential
equation (PDE) constraints under uncertainty: the state in our setting is
the solution of a parabolic PDE with a random thermal diffusion
coefficient, steered by a control function. To account for the presence of
uncertainty in the optimal control problem, the objective function is
composed with a risk measure. We focus on two risk measures, both
involving high-dimensional integrals over the stochastic variables: the
expected value and the (nonlinear) entropic risk measure. The
high-dimensional integrals are computed numerically using specially
designed QMC methods and, under moderate assumptions on the input random field, the error rate
is shown to be essentially linear, independently of the stochastic
dimension of the problem -- and thereby superior to ordinary Monte Carlo
methods. Numerical results demonstrate the effectiveness of our method.
\end{abstract}

\section{Introduction}

Many problems in science and engineering, including optimal control
problems governed by partial differential equations (PDEs), are subject to
uncertainty. If not taken into account, the inherent uncertainty of such
problems has the potential to render worthless any solutions obtained using
state-of-the-art methods for deterministic problems. The careful analysis
of the uncertainty in PDE-constrained optimization is hence indispensable
and a growing field of research (see, e.g.,
\cite{Chen2,Chen3,Chen1,Surowiec3,KouriSurowiec1,KouriSurowiec2,Nobile2,Nobile1,Stadler,Andreas2,Andreas1}).

In this paper we consider the heat equation with an uncertain thermal
diffusion coefficient, modelled by a series in which a countably infinite
number of independent random variables enter affinely. By controlling the
source term of the heat equation, we aim to steer its solution
towards a desired target state. To study the effect of this randomness on the objective function, we consider two risk measures: the expected
value and the entropic risk measure, both involving integrals with respect
to the countably infinite random variables. The integrals are replaced by integrals over finitely many random variables by truncating the series that represents the input random field to a sum over finitely many terms and then approximated using quasi-Monte Carlo (QMC) methods.

QMC approximations are particularly well suited for optimization since
they preserve convexity due to their nonnegative (equal) cubature
weights. Moreover, for sufficiently smooth integrands it is possible to
construct QMC rules with error bounds not depending on the number of
stochastic variables while attaining faster convergence rates compared to
Monte Carlo methods. For these reasons QMC methods have been very
successful in applications to PDEs with random coefficients (see, e.g.,
\cite{vandewalle1,dick,herrmann2,GGKSS2019,gilbert,harbrecht,herrmann3,herrmann1,kuonuyenssurvey,KSSSU2015,KSS2012,vandewalle2,scheichl,Schwab}) and especially
in PDE-constrained optimization under uncertainty, see
\cite{GKKSS2019,Andreas3}. In \cite{KunothSchwab} the authors derive
regularity results for the saddle point operator, which fall within the
same framework as the QMC approximation of affine parametric operator
equation setting considered in \cite{Schwab}.

This paper builds upon our previous work \cite{GKKSS2019}.
The novelty lies in the use and analysis of parabolic PDE constraints in conjunction with the nonlinear entropic risk measure, which inspired the development of an error analysis that is applicable in separable Banach spaces and thus discretization invariant. 
Solely based on regularity assumptions, our novel error analysis covers a very general class of problems. 
Specifically, we extend QMC error bounds in the literature (see, e.g.,
\cite{dick2013high,kuonuyenssurvey}) to separable Banach spaces. A crucial part of our new work is the regularity analysis of the entropic risk measure, which is used to prove our main theoretical results about error estimates and convergence rates for the dimension truncation and the QMC errors. We then apply these new bounds to assess the total errors in the optimal control problem under uncertainty.

The structure of this paper is as follows. The parametric weak formulation
of the PDE problem is given in Section~\ref{sec:problem}. The
corresponding optimization problem is discussed in
Section~\ref{sec:optimal}, with linear risk measures considered in
Subsection~\ref{sec:linear}, the entropic risk measures in
Subsection~\ref{sec:entropic}, and optimality conditions in
Subsection~\ref{sec:opt-cond}. While the regularity of the adjoint PDE
problem is the topic of Section~\ref{sec:reg-adj}, the regularity analysis for
the entropic risk measure is addressed in Section~\ref{sec:reg-ent}.
Section~\ref{sec:analysis} contains the main error analysis 
of this paper. Subsection~\ref{sec:trunc} covers the truncation error and
Subsection~\ref{sec:qmc} analyzes the QMC integration error. Our
approach differs from most studies of QMC in the literature insofar as we develop
the QMC and dimension truncation error theory for the full PDE solutions
(with respect to an appropriately chosen function space norm) instead of
considering the composition of the PDE solution with a linear functional.
In Section~\ref{sec:num} we confirm our theoretical findings with
supporting numerical experiments. Section~\ref{sec:conclusion} is a
concluding summary of this paper.

\section{Problem formulation} \label{sec:problem}

Let $D\subset\mathbb{R}^d$, $d\in\{1,2,3\}$, denote a bounded physical
domain with Lipschitz boundary, let $I := [0,T]$ denote the time interval
with finite time horizon $0<T<\infty$,  and let
$U:=[-\frac{1}{2},\frac{1}{2}]^{\bbN}$ denote a space of parameters. The components of the sequence $\bsy\in U$ are realizations of independently and identically distributed uniform random variables in $[-\tfrac12,\tfrac12]$, and the corresponding probability measure is
$$
\mu({\rm d}\bsy)=\bigotimes_{j\geq 1}{\rm d}y_j={\rm d}\bsy.
$$
Let
\begin{align} \label{eq:axy}
  a^{\bsy}(\bsx,t) := a_0(\bsx,t) + \sum_{j\geq 1} y_j\,\psi_j(\bsx,t),
  \qquad  \bsx\in D,\quad \bsy\in U,\quad t\in I,
\end{align}
be an uncertain (thermal) diffusion coefficient, where we assume (i) for
a.e.~$t\in I$ we have $a_0(\cdot,t) \in L^\infty(D)$,  $\psi_j(\cdot,t)\in
L^\infty(D)$ for all $j\geq 1$, and that $(\sup_{t\in
I}\|\psi_j(\cdot,t)\|_{L^\infty(D)})_{j\geq 1}\in\ell^1$; (ii) $t \mapsto
a^\bsy(\bsx,t)$ is measurable on $I$; (iii) uniform ellipticity: there
exist positive constants $a_{\min}$ and $a_{\max}$ such that
$0<a_{\min}\leq a^\bsy(\bsx,t)\leq a_{\max}<\infty$ for all $\bsx\in D$,
$\bsy\in U$ and a.e.~$t\in I$. Time-varying diffusion coefficients occur
e.g., in finance or cancer tomography. However, the presented setting clearly also includes time-constant diffusion coefficients,  i.e.,
$a^\bsy(\bsx,t) = a^\bsy(\bsx)\, \forall t \in I$.

We consider the heat equation over the time interval $I=[0,T]$,  given by the partial differential equation (PDE)
\begin{align} \label{eq:model}
 \begin{cases}
 \frac{\partial}{\partial t} u^{\bsy}(\bsx,t)-\nabla\cdot \big(a^{\bsy}(\bsx,t)\nabla u^{\bsy}(\bsx,t)\big)
 = z(\bsx,t), &\bsx\in D,\quad\;\, t\in I,\\
 u^{\bsy}(\bsx,t) = 0, &\bsx\in \partial D, \quad t\in I,\\
 u^{\bsy}(\bsx,0) = u_0(\bsx), &\bsx\in D,
 \end{cases}
\end{align}
for all $\bsy\in U$. Here $z(\bsx,t)$ is the control and $u_0 \in L^2(D)$ denotes the initial heat distribution. We denote the input functions collectively by $f := (z, u_0)$. We have imposed homogeneous Dirichlet boundary conditions.

Given a target state $\widehat{u}(\bsx,t)$, we will study the problem of minimizing the following objective function:
\begin{align}\label{eq:objective}
 \widetilde{J}(u,z) := \calR \Big(\frac{\alpha_1}{2}\|u^\bsy - \widehat{u}\|^2_{L^2(V;I)}
 + \frac{\alpha_2}{2} \|u^\bsy(\cdot,T) - \widehat{u}(\cdot,T)\|_{L^2(D)}^2 \Big)\nonumber
 \\+ \frac{\alpha_3}{2}\|z\|_{L^2(V';I)}^2\,,
\end{align}
subject to the PDE \eqref{eq:model} and constraints on the control to be defined later in the manuscript. By $\calR$ we denote a risk measure, which is a functional that maps a set of random variables into the extended real numbers. Specifically, $\calR$ will later be either the expected value or the entropic risk measure, both involving high-dimensional integrals with respect to $\bsy$. We will first introduce a function space setting to describe the problem properly,  including the definition of the $L^2(V;I)$ and $L^2(V';I)$ norms.

\subsection{Function space setting}\label{sec:spaces}

We define $V := H_0^1(D)$ and its (topological) dual space $V' := H^{-1}(D)$, and identify $L^2(D)$ with its own dual. Let $\langle \cdot,\cdot\rangle_{V',V}$ denotes the duality pairing between $V'$ and $V$. The norm and inner product in $V$ are defined as usual by
\[
  \|v\|_V := \|\nabla v\|_{L^2(D)}, \quad
  \langle v_1,v_2\rangle_V := \langle \nabla v_1,\nabla v_2\rangle_{L^2(D)}.
\]
We shall make use of the Riesz operator $R_V\!: V\to V'$ defined by
\begin{align} \label{eq:RV}
  \langle R_V v_1, v_2 \rangle_{V',V} = \langle v_1,v_2\rangle_V
  \quad\forall\,v_1,v_2\in V,
\end{align}
as well as its inverse $R_V^{-1}\!: V'\to V$ satisfying $R_V^{-1} w = v
\Leftrightarrow w = R_V v$ for $v\in V,\,w\in V'$. It follows from
\eqref{eq:RV} that
\begin{align} \label{eq:RVinv}
  \langle w, v \rangle_{V',V} = \langle R_V^{-1}w,v\rangle_V
  \quad\forall\,v\in V,w\in V'.
\end{align}
In turn we define the inner product in $V'$ by
\[
  \langle w_1, w_2\rangle_{V'} := \langle R_V^{-1} w_1, R_V^{-1} w_2\rangle_V.
\]
The norm induced by this inner product is equal to the usual dual norm.

We use analogous notations for inner products and duality pairings between
function spaces on the space-time cylinder $D\times I$. The space
$L^2(V;I)$ consists of all measurable functions $v: I \to V$ with finite
norm
\[
  \|v\|_{L^2(V;I)} := \Big( \int_I \|v(\cdot,t)\|_{V}^2\, \rd t \Big)^{1/2}.
\]
Note that $(L^2(V;I))'=L^2(V';I)$, with the duality pairing given by
\[
  \langle w,v\rangle_{L^2(V';I),L^2(V;I)}
  = \int_I \langle w(\cdot,t),v(\cdot,t)\rangle_{V',V}\,\rd t.
\]
We extend the Riesz operator $R_V$ to $R_V: L^2(V;I) \to L^2(V';I)$ so
that
\begin{align*}
 \langle v_1,v_2\rangle_{L^2(V;I)}
 &= \int_I \langle v_1(\cdot,t),v_2(\cdot,t)\rangle_V\,\rd t
 = \int_I \big\langle R_V v_1(\cdot,t),v_2(\cdot,t)\big\rangle_{V',V}\,\rd t \\
 &= \big\langle R_V v_1,v_2\big\rangle_{L^2(V';I),L^2(V;I)} \quad\forall\, v_1,v_2\in L^2(V;I),
\end{align*}
and we extend the inverse $R_V^{-1}: L^2(V';I) \to L^2(V;I)$ analogously.

We define the space of solutions $u^\bsy$ for $\bsy\in U$ by
\begin{align*}
  \calX := \Big\{v\in L^2(V;I) : \tfrac{\partial}{\partial t}v \in L^2(V';I)\Big\},
\end{align*}
which is the space of all functions $v$ in $L^2(V;I)$ with
(distributional) derivative $\tfrac{\partial}{\partial t}v$ in
$L^2(V';I)$, and which is equipped with the (graph) norm
\begin{align*}
 \|v\|_{\calX} :=&\,
 \Big(\int_I \Big(\|v(\cdot,t)\|_{V}^2
 + \|\tfrac{\partial}{\partial t}v(\cdot,t)\|_{V'}^2 \Big)\, \rd t \Big)^{1/2} \\
 =&\, \Big(\|v\|_{L^2(V;I)}^2 + \|\tfrac{\partial}{\partial t}v\|_{L^2(V';I)}^2 \Big)^{1/2}.
\end{align*}
Finally, because there are two inputs in equation \eqref{eq:model}, namely
$z \in L^2(V';I)$ and $u_0 \in L^2(D)$, it is convenient to define the
product space $\calY := L^2(V;I) \times L^2(D)$, and its dual space
by $\calY' := L^2(V';I) \times L^2(D)$, with the norms
\begin{align*}
 \|v\|_{\calY} &:= \Big(\int_I\|v_1(\cdot,t)\|_{V}^2\,\rd t + \|v_2\|_{L^2(D)}^2\Big)^{1/2}\!\!,\\
 \|w\|_{\calY'} &:= \Big(\int_I\|w_1(\cdot,t)\|_{V'}^2\,\rd t + \|w_2\|_{L^2(D)}^2\Big)^{1/2}\!\!.
\end{align*}

In particular, we extend $\calX$ to $\calY$ as follows. For all $v \in
\calX$ we interpret $v$ as an element of $\calY$ as $v =
(v(\bsx,t),v(\bsx,0))$. This gives $\calX \subseteq \calY$. We further
know from \cite[Theorem~5.9.3]{Evans2010} that $\calX \hookrightarrow
\calC(L^2(D);I)$ and $\max_{t \in I} \|v(\cdot,t)\|_{L^2(D)} \leq
C_1(\|v\|_{L^2(V;I)} + \|\tfrac{\partial}{\partial t} v\|_{L^2(V';I)})
\leq \sqrt{2}\, C_1 \|v\|_\calX$ for $v \in \calX$, where $C_1$ depends on
$T$ only. Hence we obtain for all $v \in \calX$ that
\begin{align*}
    &\|v\|_\calY^2
    = \|v\|_{L^2(V;I) \times L^2(D)}^2
    = \|v\|_{L^2(V;I)}^2 + \|v(\cdot,0)\|_{L^2(D)}^2\\
    &\le \|v\|_{L^2(V;I)}^2 + \Big(\max_{t \in I}\|v(\cdot,t)\|_{L^2(D)}\Big)^2
    \le \|v\|_\calX^2 + 2\,C_1^2\|v\|_\calX^2
    = (1 + 2\,C_1^2)\|v\|^2_\calX,
\end{align*}
and thus we get that $\calX$ is continuously embedded into $\calY$, i.e.,
$\calX \hookrightarrow \calY$.

\subsection{Variational formulation}\label{sec:variational}

Based on these spaces, using integration by parts with respect to $\bsx$
we can write \eqref{eq:model} as a variational problem as follows. Given
the input functions $f =(z,u_0)\in \calY'$ and $\bsy \in U$, find a
function $u^{\bsy} \in \calX$ such that
\begin{align}\label{eq:weakPDE}
    b(\bsy;u^{\bsy},v) \,=\, \langle f,v \rangle_{\calY',\calY} \quad \forall\; v= (v_1,v_2) \in \calY\,,
\end{align}
where for all $w\in \calX$, $v = (v_1,v_2) \in \calY$ and $\bsy \in U$,
\begin{align} \label{eq:bilinear}
 &b(\bsy;w,v) := \langle B^\bsy w, v\rangle_{\calY',\calY} \nonumber\\
 &:= \underbrace{\int_I \big\langle \tfrac{\partial}{\partial t} w,v_1\big\rangle_{V',V}\, \rd t
 + \int_I \int_{D} \big(a^{\bsy} \nabla w \cdot \nabla v_1\big)\,\rd \bsx\,\rd t
 }_{=:\,\langle B_1^\bsy w,v_1 \rangle_{L^2(V';I),L^2(V;I)}}
 + \underbrace{\int_D w(\cdot,0)\,v_2\, \rd\bsx}_{=:\, \langle B_2^\bsy w, v_2 \rangle_{L^2(D)}}\,, %
\end{align}
\[
 \langle f,v \rangle_{\calY',\calY}
 := \int_I \langle z, v_1\rangle_{V',V}\, \rd t + \int_D u_0\, v_2\,\rd\bsx\,, \nonumber 
\]
with operators $B^\bsy: \calX \to \calY'$, $B_1^\bsy: \calX \to
L^2(V';I)$, $B_2^\bsy:\calX \to L^2(D)$, and $B^\bsy w = (B_1^\bsy w,
B_2^\bsy w)$. For better readability we have omitted the parameter
dependence $v = (v_1(\bsx,t),v_2(\bsx))$, $f = (z(\bsx,t), u_0(\bsx))$, $w
= w(\bsx,t)$ and $a^{\bsy} = a^{\bsy}(\bsx,t)$. Note that a solution of
\eqref{eq:weakPDE} automatically satisfies $u^\bsy(\cdot,0) = u_0$, as can
be seen by setting $v_1 = 0$ and allowing arbitrary $v_2$.

The parametric family of parabolic evolution operators $\{B^\bsy,\, \bsy
\in U\}$ associated with this bilinear form is a family of isomorphisms
from $\calX$ to $\calY'$, see, e.g., \cite{DautrayLions2000}. In
\cite{SchwabStevenson} a shorter proof based on the characterization of
the bounded invertibility of linear operators between Hilbert spaces is
presented, together with precise bounds on the norms of the operator and
its inverse: there exist constants $0 < \beta_1 \leq \beta_2 < \infty$
such that
\begin{align} \label{eq:beta}
    \sup_{\bsy \in U} \|(B^\bsy)^{-1}\|_{\calY' \to \calX} \le \frac{1}{\beta_1}
    \quad
    \text{and}
    \quad
    \sup_{\bsy \in U} \|B^\bsy\|_{\calX \to \calY'} \le \beta_2\,,
\end{align}
where $\beta_1 \geq \tfrac{\min\{a_{\min}
a_{\max}^{-2},a_{\min}\}}{\sqrt{2\max\{a_{\min}^{-2},1\}+\varrho^2}}$ and
$\beta_2 \leq \sqrt{2\max\{1,a_{\max}^2\}+\varrho^2}$ with $\varrho :=
\sup\limits_{w \in \calX} \tfrac{\|w(\cdot,0)\|_{L^2(D)}}{\|w\|_{\calX}}$,
and hence for all $\bsy \in U$ we have the \emph{a priori} estimate
\begin{align} \label{eq:u-apriori}
    \|u^{\bsy}\|_{\calX}
    \le \frac{\|f\|_{\calY'}}{\beta_1} =
    \frac{1}{\beta_1} \|(z,u_0)\|_{\calY'}
    = \frac{1}{\beta_1} \big( \|z\|^2_{L^2(V';I)} + \|u_0\|^2_{L^2(D)} \big)^{1/2}\,.
\end{align}

With $\bbN_0 := \{0,1,2,\ldots\}$, let $\bsnu\in\mathbb{N}_0^\infty$
denote a multi-index, and define ${\supp}(\bsnu):=\{j\ge 1 :
\nu_j\neq 0\}$ and $|\bsnu| := \sum_{j\ge 1} \nu_j$. In the sequel, we
shall consider the set $\mathscr{F}:=\{\bsnu\in\mathbb N_0^\infty:|{\rm
supp}(\bsnu)|<\infty\}$ of multi-indices with finite support. We use the
notation $\partial^\bsnu_\bsy := \prod_{j\ge 1} (\partial/\partial
y_j)^{\nu_j}$ to denote the mixed partial derivatives with respect to
$\bsy$. For any sequence of real numbers $\bsb = (b_j)_{j\geq1}$, we
define $\bsb^\bsnu := \prod_{j\ge 1} b_j^{\nu_j}$.

The following regularity result for the state $u^\bsy$ was proved in
\cite{KunothSchwab}.

\begin{lemma}\label{lem:statebound}
Let $f=(z,u_0)\in \calY'$. For all $\bsnu\in\mathscr{F}$ and all $\bsy\in
U$, we have
\begin{align}\label{eq:statebound}
 \|\partial^{\bsnu}_{\bsy}u^{\bsy}\|_{\calX}
 \le \frac{\|f\|_{\calY'}}{\beta_1} \,
 |\bsnu|!\,\bsb^{\bsnu},
\end{align}
where $\beta_1$ is as described in \eqref{eq:beta} and the sequence $\bsb
= (b_j)_{j\ge 1}$ is defined by
\begin{align} \label{eq:bj}
 b_j := \frac{1}{\beta_1}\,\sup_{t\in I}\|\psi_j(\cdot,t)\|_{L^{\infty}(D)}.
\end{align}
\end{lemma}

For our later derivation of the optimality conditions for the optimal
control problem, it is helpful to write the variational form of the PDE
\eqref{eq:weakPDE} as an operator equation using \eqref{eq:bilinear}:
\begin{align} \label{eq:opPDE}
 B^\bsy u^\bsy \,=\, (B_1^\bsy u^\bsy, B_2^\bsy u^\bsy) \,=\, (z, u_0) \qquad \text{in }\calY'\,,
\end{align}
with $B_1^\bsy: \calX \to L^2(V';I)$ and $B_2^\bsy: \calX \to L^2(D)$
given by
\[ 
 B_1^\bsy = \Lambda_1 B^\bsy \quad \text{and} \quad B_2^\bsy = \Lambda_2 B^\bsy\,,
\] 
where $\Lambda_1: \calY' \to L^2(V';I)$ and $\Lambda_2: \calY' \to L^2(D)$
are the restriction operators defined, for any $v = (v_1,v_2) \in \calY'$,
by
\begin{align*}
 \Lambda_1 (v_1,v_2) := v_1 \quad \text{ and } \quad
 \Lambda_2 (v_1,v_2) := v_2\,.
\end{align*}

For the definition of a meaningful inverse of the operators $B_1^\bsy$ and
$B_2^\bsy$, we first define the trivial extension operators
$\Xi_1:L^2(V';I) \to \calY'$ and $\Xi_2: L^2(D) \to \calY'$, for any $v_1
\in L^2(V';I)$ and $v_2\in L^2(D)$, by
\begin{align*}
 \Xi_1 v_1 := (v_1,0) \quad \text{ and } \quad
 \Xi_2 v_2 := (0,v_2)\,.
\end{align*}
We observe that $P_1 := \Xi_1\Lambda_1$ is an orthogonal projection on the
$L^2(V';I)$-component in $\calY'$ and analogously $P_2 := \Xi_2\Lambda_2$
is an orthogonal projection on the $L^2(D)$-component in~$\calY'$. This is
verified as follows. For all $v,u\in \calY'$ it is true that
\begin{align*}
 \langle (\calI_{\calY'}-P_1) v, P_1 u \rangle_{\calY'} = 0 \quad \text{and} \quad
 \langle (\calI_{\calY'}-P_2) v, P_2 u \rangle_{\calY'} = 0\,,
\end{align*}
where $\calI_{\calY'}$ denotes the identity operator on $\calY'$. We
clearly have $\calI_{\calY'} = P_1 + P_2$. Therefore we can write any
element $v$ in $\calY'$ as $v = P_1 v + P_2 v$ in $\calY'$, and by
linearity of $(B^\bsy)^{-1}$ we get
$$
 (B^\bsy)^{-1} v = (B^\bsy)^{-1} (P_1 v + P_2 v)
 = (B^\bsy)^{-1} P_1 v + (B^\bsy)^{-1} P_2 v\,.
$$

A meaningful inverse of the operators $B^\bsy_1: \calX \to L^2(V';I)$ and
$B^\bsy_2: \calX \to L^2(D)$ are then given by $(B_1^\bsy)^\dagger:
L^2(V';I) \to \calX$ and $(B_2^\bsy)^\dagger: L^2(D) \to \calX$, defined
as
\begin{align} \label{eq:inverse}
 (B_1^\bsy)^\dagger := (B^\bsy)^{-1} \Xi_1 \quad \text{and} \quad
 (B_2^\bsy)^\dagger := (B^\bsy)^{-1} \Xi_2\,.
\end{align}
We call the operator $(B_1^\bsy)^\dagger$ the pseudoinverse of $B_1^\bsy$
and the operator $(B_2^\bsy)^\dagger$ the pseudoinverse of $B_2^\bsy$.
Clearly, the pseudoinverse operators are linear and bounded operators.

\begin{lemma}\label{lemma:lemmaMPI}
The pseudoinverse operators $(B_1^\bsy)^\dagger$ and $(B_2^\bsy)^\dagger$
defined by \eqref{eq:inverse} satisfy
\begin{align} \label{eq:identity1}
 \calI_{L^2(V';I)} &= B_1^\bsy (B_1^\bsy)^\dagger\,, \quad
 \calI_{L^2(D)} = B_2^\bsy (B_2^\bsy)^\dagger\,, \quad \text{and} \nonumber \\
 \calI_\calX &= (B_1^\bsy)^\dagger B_1^\bsy + (B_2^\bsy)^\dagger B_2^\bsy,
\end{align}
which are the identity operators on $L^2(V';I)$, $L^2(D)$, and $\calX$,
respectively.
\end{lemma}

\begin{proof}
From the definition of various operators, we have
\begin{align*}
 B_1^\bsy (B_1^\bsy)^\dagger
 &= \Lambda_1 B^\bsy (B^\bsy)^{-1} \Xi_1
 = \Lambda_1 \calI_{\calY'} \Xi_1 = \Lambda_1 \Xi_1 = \calI_{L^2(V';I)}\,,
 \\
 B_2^\bsy (B_2^\bsy)^\dagger
 &= \Lambda_2 B^\bsy (B^\bsy)^{-1} \Xi_2
 = \Lambda_2 \calI_{\calY'} \Xi_2 = \Lambda_2 \Xi_2 = \calI_{L^2(D)}\,,
 \\
 (B_1^\bsy)^\dagger B_1^\bsy + (B_2^\bsy)^{\dagger} B_2^\bsy
 &= (B^\bsy)^{-1} \Xi_1 \Lambda_1 B^\bsy +  (B^\bsy)^{-1} \Xi_2 \Lambda_2 B^\bsy \\
 &= (B^\bsy)^{-1} (P_1 + P_2) B^\bsy
 = (B^\bsy)^{-1} \calI_{\calY'} B^\bsy
 = \calI_{\calX}\,,
\end{align*}
as required.  
\end{proof}

\begin{lemma}
For $\bsy\in U$ and given $(z,u_0)\in \calY'$, the solution $u^\bsy$ of
the operator equation \eqref{eq:opPDE} can be written as
\begin{align}\label{eq:splitinvB}
 u^\bsy \,=\, (B^\bsy)^{-1}(z,u_0) \,=\,
 (B_1^\bsy)^{\dagger}z + (B_2^\bsy)^{\dagger} u_0 \quad \text{in } \calX\,.
\end{align}
\end{lemma}

\begin{proof}
From \eqref{eq:identity1} we have $u^\bsy
 = (B_1^\bsy)^{\dagger} B_1^\bsy u^\bsy + (B_2^\bsy)^{\dagger} B_2^\bsy u^\bsy
 = (B_1^\bsy)^{\dagger}z + (B_2^\bsy)^{\dagger} u_0$,
as required.  
\end{proof}

\subsection{Dual problem}

In the following we will need the dual operators $(B^\bsy)'$,
$(B^\bsy_1)'$ and $(B^\bsy_2)'$ of $B^\bsy$, $B^\bsy_1$ and $B^\bsy_2$,
respectively, which are formally defined by
\begin{align*}
    \langle w, (B^\bsy)'v\rangle_{\calX,\calX'} &\,:=\, \langle B^\bsy w, v\rangle_{\calY',\calY}\\
    \langle w, (B_1^\bsy)'v_1\rangle_{\calX,\calX'} &\,:=\, \langle B_1^\bsy w, v_1\rangle_{L^2(V';I),L^2(V;I)}\\
    \langle w, (B_2^\bsy)'v_2\rangle_{\calX,\calX'} &\,:=\, \langle B_2^\bsy w, v_2\rangle_{L^2(D)}
\end{align*}
for all $w\in \calX$, $v = (v_1,v_2) \in \calY$ and $\bsy \in U$, with
$(B^\bsy)'v = (B_1^\bsy)'v_1 + (B_2^\bsy)'v_2$.

The dual problem to \eqref{eq:weakPDE} (or equivalently \eqref{eq:opPDE})
is as follows. Given the input function $f_{\rm dual} \in \calX'$ and
$\bsy \in U$, find a function $q^\bsy = (q_1^\bsy,q_2^\bsy)\in \calY$ such
that
\begin{align}\label{eq:weakdual}
    \langle w, (B^\bsy)'q^\bsy \rangle_{\calX,\calX'} = \langle w, f_{\rm dual} \rangle_{\calX,\calX'}\quad \forall
    w\in  \calX,
\end{align}
or in operator form
$
(B^\bsy)'q^\bsy =f_{\rm dual},
$
which has the unique solution
$
q^\bsy= \big((B^\bsy)'\big)^{-1} f_{\rm dual}\,.
$

Existence and uniqueness of the solution of the dual problem follow
directly from the bounded invertibility of $B^\bsy$. We know that its
inverse, $(B^\bsy)^{-1}$, is a bounded linear operator and thus the dual
of $(B^\bsy)^{-1}$ is (uniquely) defined (see, e.g., \cite[Theorem~1 and
Definition~1, Chapter~VII]{Yosida}). The operator $(B^\bsy)^{-1}$ and its
dual operator $((B^\bsy)^{-1})' =((B^\bsy)')^{-1}$ are equal in their
operator norms (see, e.g., \cite[Theorem~2, Chapter~VII]{Yosida}), i.e.,
the operator norms of the dual operator $(B^\bsy)'$ and its inverse are
bounded by the constants $\beta_2$ and $\tfrac{1}{\beta_1}$
in~\eqref{eq:beta}.

Applying integration by parts with respect to the time variable in
\eqref{eq:bilinear}, the left-hand side of the dual problem
\eqref{eq:weakdual} can be written as
\begin{align}\label{eq:vardual}
 &\langle w, (B^\bsy)'q^\bsy\rangle_{\calX,\calX'}
 = \langle B^\bsy w, q^\bsy\rangle_{\calY',\calY}  \nonumber \\
 &= \bigg(\int_I \langle w,-\tfrac{\partial}{\partial t}q_1^\bsy\rangle_{V,V'}\, \rd t
 + \int_I \int_D (a^\bsy \nabla w \cdot \nabla q_1^\bsy)\,\rd \bsx\,\rd t \nonumber \\
 & \quad
 + \int_D w(\cdot,T)\,q_1^\bsy(\cdot,T)\,\rd \bsx - \int_D w(\cdot,0)\,q_1^\bsy(\cdot,0)\,\rd \bsx \bigg)
 + \int_D w(\cdot,0)\, q_2^\bsy\,\rd \bsx \\
 &= \big\langle w,(B_1^\bsy)'q_1^\bsy \big\rangle_{\calX,\calX'} + \big\langle w,(B_2^\bsy)'q_2^\bsy \big\rangle_{\calX,\calX'}\,.
 \nonumber
\end{align}

We may express the solution $q^\bsy = (q_1^\bsy,q_2^\bsy)\in\calY$ of the
dual problem \eqref{eq:weakdual} in terms of the dual operators of the
pseudoinverse operators $(B_1^\bsy)^\dagger$ and $(B_2^\bsy)^\dagger$.
This is true because we get an analogous result to
Lemma~\ref{lemma:lemmaMPI} in the dual spaces.

\begin{lemma} \label{lemma:lemmaMPIdual}
The dual operators $((B_1^\bsy)^\dagger)'$ and $((B_2^\bsy)^\dagger)'$ of
the pseudoinverse operators defined in \eqref{eq:inverse} satisfy
\begin{align} \label{eq:identity2}
 \calI_{L^2(V;I)} &= ((B_1^\bsy)^\dagger)'(B_1^\bsy)'\,, \quad
 \calI_{L^2(D)} = ((B_2^\bsy)^\dagger)'(B_2^\bsy)' \,, \quad \text{and} \nonumber\\
 \calI_{\calX'} &= (B_1^\bsy)'((B_1^\bsy)^\dagger)' + (B_2^\bsy)' ((B_2^\bsy)^\dagger)'\,,
\end{align}
which are the identity operators on $L^2(V;I)$, $L^2(D)$ and $\calX'$,
respectively.
\end{lemma}

\begin{proof}
For all $v_1 \in L^2(V';I)$, $w_1 \in L^2(V;I)$, $v_2,w_2\in L^2(D)$, it
follows from \eqref{eq:identity1} that
\begin{align*}
 \langle v_1,w_1 \rangle_{L^2(V';I),L^2(V;I)}
 &= \big\langle B_1^\bsy (B_1^\bsy)^\dagger v_1, w_1 \big\rangle_{L^2(V';I),L^2(V;I)}\\
 &= \big\langle v_1,((B_1^\bsy)^\dagger)' (B_1^\bsy)' w_1 \big\rangle_{L^2(V';I),L^2(V;I)}\,, \mbox{ and} \\
 \langle v_2,w_2 \rangle_{L^2(D)}
 &= \big\langle B_2^\bsy (B_2^\bsy)^\dagger v_2,w_2 \big\rangle_{L^2(D)}
 \!=\! \big\langle v_2, ((B_2^\bsy)^\dagger)' (B_2^\bsy)' w_2 \big\rangle_{L^2(D)}.
\end{align*}
Similarly, for all $v\in\calX$ and $w\in\calX'$ we have
\begin{align*}
 \langle v,w \rangle_{\calX,\calX'}
 &= \big\langle \big((B_1^\bsy)^\dagger B_1^\bsy + (B_2^\bsy)^\dagger B_2^\bsy\big) v, w \big\rangle_{\calX,\calX'}\\
 &= \big\langle (B_1^\bsy)^{\dagger} B_1^\bsy v, w \big\rangle_{\calX,\calX'}
 + \big\langle (B_2^\bsy)^{\dagger} B_2^\bsy v, w \big\rangle_{\calX,\calX'}  \\
 &= \big\langle v, (B_1^\bsy)'((B_1^\bsy)^\dagger)' w \big\rangle_{\calX,\calX'}
 + \big\langle v, (B_2^\bsy)'((B_2^\bsy)^\dagger)' w \big\rangle_{\calX,\calX'} \\
 &= \langle v, \big((B_1^\bsy)'((B_1^\bsy)^\dagger)'
 + (B_2^\bsy)'((B_2^\bsy)^\dagger)'\big) w \rangle_{\calX,\calX'}\,.
\end{align*}
This completes the proof.  
\end{proof}

\begin{lemma}\label{lemma:dual_sol} Given the input function $f_{\rm dual}
\in \calX'$ and $\bsy \in U$, the (unique) solution of the dual problem
\eqref{eq:weakdual} is given by
\begin{equation}\label{eq:dualexplicit}
  q^\bsy = (q^\bsy_1, q^\bsy_2)
  = \big( ((B_1^\bsy)^\dagger)' f_{\rm dual}, ((B_2^\bsy)^\dagger)' f_{\rm dual}\big)
  \quad\mbox{in $\calY$}\,.
\end{equation}
\end{lemma}

\begin{proof}
Existence and uniqueness follow from the bounded invertibility of
$(B^\bsy)^\prime$, see Subsection~\ref{sec:variational}. Thus, we only
need to verify that \eqref{eq:dualexplicit} solves the dual problem
\eqref{eq:weakdual}. It follows from \eqref{eq:identity2} that
\begin{align*}
 f_{\rm dual}
 &= \big((B_1^\bsy)'((B_1^\bsy)^\dagger)' + (B_2^\bsy)' ((B_2^\bsy)^\dagger)'\big) f_{\rm dual} \\
 &= (B_1^\bsy)'((B_1^\bsy)^\dagger)' f_{\rm dual} + (B_2^\bsy)' ((B_2^\bsy)^\dagger)' f_{\rm dual}\\
 &= (B_1^\bsy)' q_1^\bsy + (B_2^\bsy)' q_2^\bsy
 = (B^\bsy)' q^\bsy\,,
\end{align*}
as required.  
\end{proof}

We will see in the next section that, with the correct choice of the
right-hand side $f_{\rm dual}$, the gradient of the objective
function \eqref{eq:objective} can be computed using the solution $q^\bsy$
of the dual problem.

\section{Parabolic optimal control problems under uncertainty with control
constraints} \label{sec:optimal}

The presence of uncertainty in the optimization problem requires the
introduction of a risk measure $\mathcal{R}$ that maps the random variable
objective function (see \eqref{eq:Phi} below) to the extended real
numbers. Let $(\Omega, \mathcal{A}, \mathbb{P})$ be a complete
probability space. A functional $\calR : L^p(\Omega, \mathcal{A},
\mathbb{P}) \to \mathbb{R} \cup \{\infty\}$, for $p \in [1,\infty)$, is
said to be a \emph{coherent} risk measure \cite{Artzner} if for
$X,\tilde{X} \in L^p(\Omega, \mathcal{A}, \mathbb{P})$ we have
\begin{enumerate}
\item [(1)\!] Convexity: $\mathcal{R}(\lambda X + (1-\lambda)
    \tilde{X}) \leq \lambda \mathcal{R}(X) + (1-\lambda)
    \mathcal{R}(\tilde{X})$ for all $\lambda \in [0,1]$.
\item [(2)\!] Translation equivariance: $\mathcal{R}(X+c) =
    \mathcal{R}(X) + c$ for all $c \in \mathbb{R}$.
\item [(3)\!] Monotonicity: If $X \leq \tilde{X}$ $\mathbb{P}$-a.e.
    then $\mathcal{R}(X) \leq \mathcal{R}(\tilde{X})$.
\item [(4)\!] Positive homogeneity: $\mathcal{R}(tX) = t
    \mathcal{R}(X)$ for all $t\geq 0$.
\end{enumerate}
Coherent risk measures are popular as numerous desirable properties can be
derived from the above conditions (see, e.g., \cite{KouriSurowiec1} and
the references therein).
However, it can be shown (see \cite[Theorem 1]{KouriSurowiec2}) that the
only coherent risk measures that are Fr\'echet differentiable are linear
ones. The expected value has all of these properties, but is risk-neutral.
In order to address also risk-averse problems we focus on the (nonlinear)
entropic risk measures, which are risk-averse, Fr\'echet differentiable,
and satisfy the conditions (1)--(3) above, i.e., they are not positively
homogeneous (and thus not coherent). Risk measures satisfying (2) and (3)
are called monetary risk measures, and a monetary risk measure that also
satisfies (1) is called a convex risk measure (see \cite{FoellmerSchied}).

In this section we will first discuss the required conditions on the
risk measure $\mathcal{R}$ under which the optimal control problem
has a unique solution. We will then present two classes of risk measures
that satisfy these conditions, namely the linear risk measures that
include the expected value, and the entropic risk measures. Finally we
derive necessary and sufficient optimality conditions for the optimal
control problem with these two risk measures.
We assume that the target state $\widehat{u}$ belongs to $\calX$
and that the constants $\alpha_{1}, \alpha_{2}$ are nonnegative with
$\alpha_1+\alpha_2
>0$ and $\alpha_{3}>0$.
Then we consider the following problem: minimize $\widetilde{J}(u,z)$
defined in \eqref{eq:objective}
subject to the parabolic PDE \eqref{eq:model} and constraints on the
control
\begin{align}
 z \in \mathcal Z\label{eq:boxconstraints}
\end{align}
with $\mathcal Z$ being nonempty, bounded, closed and convex.

We want to analyze the problem in its reduced form, i.e., expressing the
state $u^\bsy = (B^\bsy)^{-1} (z,u_0)$ in \eqref{eq:objective} in
terms of the control $z$. This reformulation is possible because of the
bounded invertibility of the operator $B^\bsy$ for every $\bsy \in U$, see
Section~\ref{sec:variational} and the references therein. We therefore
introduce an alternative notation $u(z) = (u^\bsy(z))(\bsx,t) =
u^\bsy(\bsx,t)$. (Of course $u^\bsy$ depends also on $u_0$, but we can
think of $u_0$ as fixed, and therefore uninteresting.) The reduced problem
is then to minimize
\begin{align}  \label{eq:J-reduced}
 J(z) := \widetilde{J}\big(u(z),z\big)
 =
 \calR \Big(
 \frac{\alpha_1}{2}\big\|u^\bsy(z)-\widehat{u}\big\|_{L^2(V;I)}^2
 &+ \frac{\alpha_2}{2}\big\|E_T\big(u^\bsy(z)-\widehat{u}\big)\big\|_{L^2(D)}^2
 \Big) \nonumber \\
 &+\frac{\alpha_3}{2}\|z\|_{L^2(V';I)}^2 ,
\end{align}
where $E_T\!:\calX\to L^2(D)$ is the bounded linear operator defined by
$v\mapsto v(\cdot,T)$ for some fixed terminal time $T>0$.

Defining
\begin{align}\label{eq:Phi}
 \Phi^\bsy(z) := \frac{\alpha_1}{2} \big\|(B^\bsy)^{-1}(z,u_0)-\widehat{u}\big\|^2_{L^2(V;I)}
 + \frac{\alpha_2}{2} \big\|E_T\big((B^\bsy)^{-1}(z,u_0) - \widehat{u}\big)\|^2_{L^2(D)},
\end{align}
we can equivalently write the reduced problem as
\begin{align}\label{eq:reduced_OCuU}
\min_{z \in \mathcal{Z}} \Big( \calR(\Phi^\bsy(z)) + \frac{\alpha_3}{2}\|z\|_{L^2(V';I)}^2 \Big)\,.
\end{align}

With the uniformly boundedly invertible forward operator $B^\bsy$,
our setting fits into the abstract framework of \cite{KouriSurowiec1}
where the authors derive existence and optimality conditions for
PDE-constrained optimization under uncertainty. In particular, the forward
operator $B^\bsy$, the regularization term
$\tfrac{\alpha_3}{2}\|z\|_{L^2(V';I)}^2$ and the random variable
tracking-type objective function $\Phi^\bsy$ satisfy the assumptions of
\cite[Proposition 3.12]{KouriSurowiec1}. In order to present the result
about the existence and uniqueness of the solution of
\eqref{eq:reduced_OCuU}, which is based on \cite[Proposition
3.12]{KouriSurowiec1}, we recall some definitions from convex analysis
(see, e.g., \cite{KouriSurowiec1} and the references therein):
A functional $\mathcal{R}: L^p(\Omega,\mathcal{A},\mathbb{P}) \to
\mathbb{R} \cup \{\infty\}$ is called \emph{proper} if
$\mathcal{R}(X)
> -\infty$ for all $X \in L^p(\Omega,\mathcal{A},\mathbb{P})$ and ${{\rm
dom}}(\mathcal{R}) := \{X \in L^p(\Omega,\mathcal{A},\mathbb{P}):
\mathcal{R}(X) < \infty\} \neq \emptyset$; it is called \emph{lower
semicontinuous} or \emph{closed} if its epigraph ${\rm epi}(\mathcal{R})
:= \{(X,\alpha) \in L^p(\Omega,\mathcal{A},\mathbb{P}) \times \mathbb R :
\mathcal{R}(X) \leq \alpha\}$ is closed in the product topology
$L^p(\Omega,\mathcal{A},\mathbb{P}) \times \mathbb R$.

\begin{lemma}\label{lem:exist_unique}
Let $\alpha_1,\alpha_2 \geq 0$ and $\alpha_3>0$ with $\alpha_1+\alpha_2 >0$ and let $\mathcal R$ be proper, closed, convex and monotonic, then there exists a unique solution of \eqref{eq:reduced_OCuU}.
\end{lemma}

\begin{proof}
The existence of the solution follows directly from \cite[Proposition
3.12]{KouriSurowiec1}. We thus only prove the strong convexity of the
objective function, which implies strict convexity and hence uniqueness of
the solution. Clearly $\frac{\alpha_3}{2}\|z\|^2_{L^2(V';I)}$ is strongly
convex. Since the sum of a convex and a strongly convex function is
strongly convex it remains to show the convexity of $\mathcal
R(\Phi^\bsy(z))$. By the linearity and the bounded invertibility of the
linear forward operator $B^\bsy$, the tracking-type objective functional
$\Phi^\bsy(z)$ is quadratic in $z$ and hence convex, i.e., we have for
$z,\tilde{z} \in L^2(V';I)$ and $\lambda \in [0,1]$ that
$\Phi^\bsy(\lambda z + (1-\lambda) \tilde{z}) \leq \lambda\Phi^\bsy(z) +
(1-\lambda)\Phi^\bsy(\tilde{z})$. Then, by the monotonicity and the
convexity of the risk measure $\mathcal R$ we get $ \mathcal
R(\Phi^\bsy(\lambda z + (1-\lambda) \tilde{z})) \leq \mathcal
R(\lambda\Phi^\bsy(z) + (1-\lambda)\Phi^\bsy(\tilde{z})) \leq \lambda
\mathcal R( \Phi^\bsy(z)) + (1-\lambda)\mathcal R(\Phi^\bsy(\tilde{z}))\,,
$ as required.  
\end{proof}

\subsection{Linear risk measures, including the expected value}
\label{sec:linear}

First we derive a formula for the Fr\'{e}chet derivative of
\eqref{eq:J-reduced} when $\calR$ is linear, which includes the
special case $\calR(\cdot) = \int_U (\cdot)\,\rd \bsy$.

\begin{lemma}\label{lemma:gradient}
Let $\calR$ be linear. Then the Fr\'{e}chet derivative
of~\eqref{eq:J-reduced} as an element of $(L^2(V';I))'=L^2(V;I)$ is given
by
\begin{align}
 J'(z)
 &\,=\, \calR\Big(
 \big((B_1^\bsy)^{\dagger}\big)' \big(\alpha_1 R_V + \alpha_2 E_T'E_T\big) \big(u^\bsy(z)-\widehat{u}\big)
 \Big)
 +\alpha_3 R_V^{-1} z \label{eq:Frechet}
\end{align}
for $z\in L^2(V';I)$.
\end{lemma}

\begin{proof}
For $z,\delta\in L^2(V';I)$, we can write
\begin{align*}
 &J(z+\delta)
 = \calR\Big(\frac{\alpha_1}{2} \big\|u^\bsy(z+\delta)-u^\bsy(z)+u^\bsy(z)-\widehat{u}\big\|_{L^2(V;I)}^2 \\
 &\qquad +\frac{\alpha_2}{2} \big\| E_T\big( u^\bsy(z+\delta)-u^\bsy(z)+u^\bsy(z)-\widehat{u}\big)\big\|_{L^2(D)}^2
 \Big)
  +\frac{\alpha_3}{2} \|z+\delta\|_{L^2(V';I)}^2 \\
 &= \calR\Big(\frac{\alpha_1}{2}  \big\| (B_1^\bsy)^{\dagger}\delta + \big(u^\bsy(z)-\widehat{u}\big)\big\|_{L^2(V;I)}^2 \\
 &\qquad +\frac{\alpha_2}{2}  \big\| E_T(B_1^\bsy)^{\dagger}\delta + E_T \big(u^\bsy(z)-\widehat{u}\big)\big\|_{L^2(D)}^2
 \Big)
 +\frac{\alpha_3}{2} \|z+\delta\|_{L^2(V';I)}^2,
\end{align*}
where we used \eqref{eq:splitinvB} to write $u^\bsy(z+\delta)-u^\bsy(z) =
[(B_1^\bsy)^{\dagger}(z+\delta) + (B_2^\bsy)^{\dagger}u_0] -
[(B_1^\bsy)^{\dagger}(z) + (B_2^\bsy)^{\dagger}u_0] =
(B_1^\bsy)^{\dagger}\delta$. Expanding the squared norms using $\|v+w\|^2
= \langle v+w,v+w\rangle = \|v\|^2 + 2\langle v,w\rangle + \|w\|^2$, we
obtain
\begin{align*}
 J(z+\delta)
 \,=\, J(z)+ (\partial_z J(z))\,\delta + \mathrm{o}(\delta),
\end{align*}
with the Fr\'{e}chet derivative $\partial_z J(z) : L^2(V';I)\to \mathbb{R}$
defined by
 \begin{align*}
 (\partial_z J(z))\,\delta
 &:= \calR\Big(\alpha_1
 \overbrace{\big\langle  (B_1^\bsy)^{\dagger}\delta , u^\bsy(z)-\widehat{u} \big\rangle_{L^2(V;I)}}^{=:\,{\rm Term}_1} \\
 &\quad\qquad +\alpha_2
 \underbrace{\big\langle E_T(B_1^\bsy)^{\dagger}\delta , E_T\big(u^\bsy(z)-\widehat{u}\big) \big\rangle_{L^2(D)}}_{=:\,{\rm Term}_2}
 \Big)
 +\alpha_3 \underbrace{\langle z,\delta\rangle_{L^2(V';I)}}_{=:\,{\rm Term}_3}.
\end{align*}

It remains to simplify the three terms.
Using the extended Riesz operator $R_V\!: L^2(V;I) \to L^2(V';I)$, we have
\begin{align*}
 {\rm Term}_1
 &= \big\langle u^\bsy(z)-\widehat{u} , (B_1^\bsy)^{\dagger}\delta \big\rangle_{L^2(V;I)} \\
 &= \big\langle R_V\big(u^\bsy(z)-\widehat{u}\big),(B_1^\bsy)^{\dagger}\delta \big\rangle_{L^2(V';I),L^2(V;I)} \\
 &= \big\langle R_V\big(u^\bsy(z)-\widehat{u}\big),(B_1^\bsy)^{\dagger}\delta \big\rangle_{\calX',\calX} \\
 &= \big\langle \big((B_1^\bsy)^{\dagger}\big)'R_V\big(u^\bsy(z)-\widehat{u}\big),\delta \big\rangle_{L^2(V;I),L^2(V';I)},
\end{align*}
where the third equality follows since $(B_1^\bsy)^{\dagger}\delta
\in\calX \hookrightarrow L^2(V;I)$, and the fourth equality follows from
the definition of the dual operator $((B_1^\bsy)^{\dagger})' : \calX' \to
L^2(V;I)$, noting that $(L^2(V';I))' = L^2(V;I)$.

Next, using the definition of the dual operator $(E_T)':L^2(D) \to
\calX'$, we can write
\begin{align*}
 {\rm Term}_2
 &= \big\langle E_T\big(u^\bsy(z)-\widehat{u}\big), E_T(B_1^\bsy)^{\dagger}\delta \big\rangle_{L^2(D)} \\
 &= \big\langle E_T'E_T\big(u^\bsy(z)-\widehat{u}\big), (B_1^\bsy)^{\dagger}\delta \big\rangle_{\calX',\calX} \\
 &= \big\langle \big((B_1^\bsy)^{\dagger}\big)' E_T' E_T\big(u^\bsy(z)-\widehat{u}\big),\delta \big\rangle_{L^2(V;I),L^2(V';I)}.
\end{align*}

Finally, using the definition of the $L^2(V',I)$ inner product and the
extended inverse Riesz operator $R_V^{-1}\!:L^2(V';I)\to L^2(V;I)$, we
obtain
\begin{align*}
 {\rm Term}_3
 \,=\, \langle z, \delta\rangle_{L^2(V';I)}
 \,=\, \langle R_V^{-1} z, R_V^{-1} \delta\rangle_{L^2(V;I)}
 \,=\, \big\langle R_V^{-1} z,\delta \big\rangle_{L^2(V;I),L^2(V';I)}.
\end{align*}
Writing $(\partial_z J(z))\, \delta = \langle J'(z),\delta\rangle_{L^2(V;I),L^2(V';I)}$ and collecting the terms above leads to the expression for $J'(z)$ in \eqref{eq:Frechet}. 
\end{proof}

We call $J'(z)$ the gradient of $J(z)$ and show next, that $J'(z)$ can be
computed using the solution of the dual problem \eqref{eq:weakdual} with
\begin{align}\label{eq:f-dual2}
 f_{{\rm dual}} := (\alpha_1 R_V + \alpha_2 E_T'E_T)(u^\bsy - \widehat{u}) \in \calX'\,.
\end{align}
We show this first for the special case when $\calR$ is linear.

\begin{lemma} \label{lem:grad}
Let $\alpha_1,\alpha_2 \geq 0$ and $\alpha_3 > 0$, with $\alpha_1 +
\alpha_2 > 0$. Let $f = (z,u_0) \in \calY'$ and $\widehat{u} \in
\calX$. For every $\bsy \in U$, let $u^\bsy \in \calX$ be the solution of
\eqref{eq:model} and then let $q^\bsy \in \calY$ be the solution of
\eqref{eq:weakdual} with $f_{\rm dual}$ given by \eqref{eq:f-dual2}. Then
for $\calR$ linear, the gradient of \eqref{eq:J-reduced} is given as an
element of $L^2(V;I)$ by
\begin{align}\label{eq:gradient}
    J'(z) = \calR(q_1) + \alpha_3 R_V^{-1}z
\end{align}
for $z \in L^2(V';I)$.
\end{lemma}
\begin{proof}
This follows immediately from \eqref{eq:f-dual2},
Lemma~\ref{lemma:gradient} and Lemma~\ref{lemma:dual_sol}.  
\end{proof}

\begin{proposition}
Under the conditions of Lemma~\ref{lem:grad}, with $f_{\rm dual}$ given by
\eqref{eq:f-dual2}, the dual solution $q^\bsy = (q_1^\bsy,q_2^\bsy) \in
\calY$ satisfies
\[
 q_2^\bsy = q_1^\bsy(\cdot,0).
\]
Consequently, the left-hand side of \eqref{eq:weakdual} reduces to
\begin{align}\label{eq:weakdual3}
 \int_I \big\langle w,-\tfrac{\partial}{\partial t}q_1^\bsy\big\rangle_{V,V'}\, \rd t
 + \int_I \int_D \big(a^\bsy \nabla w \cdot \nabla q_1^\bsy\big)\,\rd \bsx\,\rd t
 + \int_D w(\cdot,T)\,q_1^\bsy(\cdot,T)\,\rd \bsx\,,
\end{align}
and hence $q_1^\bsy$ is the solution to
\begin{align} \label{eq:strongdual}
 \begin{cases}
 -\frac{\partial}{\partial t} q_1^{\bsy}(\bsx,t)-\nabla\cdot \big(a^{\bsy}(\bsx,t)\nabla q_1^{\bsy}(\bsx,t)\big)
 = \alpha_1 R_V\big(u^\bsy(\bsx,t)-\widehat{u}(\bsx,t)\big) \\
 q_1^{\bsy}(\bsx,t) = 0 \\
 q_1^{\bsy}(\bsx,T) = \alpha_2 \big(u^\bsy(\bsx,T)-\widehat{u}(\bsx,T)\big),
 \end{cases}
\end{align}
where the first equation holds for $\bsx \in D$, $t\in I$,  and the second equation holds for $\bsx \in \partial D$, $t\in I$, and the last equation holds for $x\in D$.
\end{proposition}

\begin{proof}
Since \eqref{eq:weakdual} holds for arbitrary $w\in\calX$, it holds in
particular for the special case
\begin{align*}
 w = w_n(\bsx,t) := \begin{cases}
 \big(1-\tfrac{nt}{T}\big)\,v(\bsx) & \text{for } t\in \left[0,\tfrac{T}{n} \right]\,,\\
 0 & \text{for } t\in \left(\tfrac{T}{n},T \right]\,,
\end{cases}
\end{align*}
with arbitrary $v \in V$. For $f_{\rm dual}$ given by
\eqref{eq:f-dual2}, the right-hand side of \eqref{eq:weakdual} becomes
\begin{align*}
 &\langle w_n, f_{{\rm dual}} \rangle_{\calX,\calX'} \\
 &= \big\langle w_n, \alpha_1 R_V (u^\bsy-\widehat{u}) \big\rangle_{\calX,\calX'}
  + \big\langle w_n(\cdot,T), \alpha_2\big(u^\bsy(\cdot,T)-\widehat{u}(\cdot,T)\big) \big\rangle_{L^2(D)} \nonumber\\
 &= \int_0^{\frac{T}{n}} \!\int_D \big(1 - \tfrac{nt}{T}\big)\,v\, \alpha_1 R_V(u^\bsy-\widehat{u}) \,\rd\bsx\, \rd t
 \,\,\, \to 0 \,\,\,\mbox{as}\,\,\, n\to\infty\,.
\end{align*}	
From \eqref{eq:vardual} the left-hand side of \eqref{eq:weakdual} is now
\begin{align*} 
 &\langle w_n, (B^\bsy)' q^\bsy \rangle_{\calX,\calX'} \\
 &= \int_0^{\frac{T}{n}} \! \big(1 - \tfrac{nt}{T}\big) \big\langle v,-\tfrac{\partial}{\partial t}q_1^\bsy\big\rangle_{V,V'}\, \rd t
  + \int_0^{\frac{T}{n}} \! \int_D \big(1 - \tfrac{nt}{T}\big)
 \big(a^\bsy \nabla v \cdot \nabla q_1^\bsy\big)\,\rd \bsx\,\rd t \nonumber \\
 & \qquad
 - \int_D v\,q_1^\bsy(\cdot,0)\,\rd \bsx
 + \int_D v\,q_2^\bsy\,\rd \bsx \\
 &\to \int_D v\,\big(q_2^\bsy - q_1^\bsy(\cdot,0)\big)\,\rd \bsx
 \quad\mbox{as}\quad n\to\infty\,.
\end{align*}
Equating the two sides, letting $n\to\infty$, and noting that $v\in
V$ is arbitrary, we conclude that necessarily $q_2^\bsy =
q_1^\bsy(\cdot,0)$.

Hence, the left-hand side of \eqref{eq:weakdual} reduces to
\eqref{eq:weakdual3}. By analogy with the weak form of \eqref{eq:model},
using the transformation $t \mapsto T-t$, we conclude that $q_1^\bsy$ is
the solution to \eqref{eq:strongdual}.  
\end{proof}

\subsection{The entropic risk measure} \label{sec:entropic}

The expected value is risk neutral. Next, we consider risk averse risk
measures such as the entropic risk measure
\begin{align*}
    \calR_{{\rm e}}(Y(\bsy)) :=
    \frac{1}{\theta} \ln \Big( \int_U \exp\big(\theta\,Y(\bsy)\big) \,\mathrm d\bsy \Big)\,,
\end{align*}
for an essentially bounded random variable $Y(\bsy)$ and some
$\theta\in(0,\infty)$. Using $\calR = \calR_{{\rm e}}$ in
\eqref{eq:J-reduced}, the optimal control problem becomes $\min_{z \in
\mathcal Z} J(z)$, with
\begin{align}\label{eq:J-entropic}
    J(z) = \frac{1}{\theta} \ln
    \Big( \int_U \exp\big(\theta\,\Phi^\bsy(z)\big) \,\mathrm d\bsy \Big)
    + \frac{\alpha_3}{2} \|z\|_{L^2(V';I)}^2\,,
\end{align}
for some $\theta\in(0,\infty)$ and $\Phi^\bsy$ defined in \eqref{eq:Phi}.

In the following we want to compute the Fr\'echet derivative of $J(z)$
with respect to $z \in L^2(V';I)$. To this end, we verify that
$\Phi^\bsy(z) \leq C < \infty$ is uniformly bounded in $\bsy \in U$ for
any $z\in L^2(V';I)$, i.e.~the constant $C>0$ is independent of $\bsy \in
U$.

\begin{lemma}\label{lemma:Phi}
Let $f = (z,u_0)\in\calY'$ and $\widehat{u} \in \calX$, and let
$\alpha_1,\alpha_2 \geq 0$ with $\alpha_1+\alpha_2 >0$. Then for all
$\bsy\in U$, the function $\Phi^\bsy$ defined by \eqref{eq:Phi} satisfies
\begin{align} \label{eq:Phi-bound}
  0\le \Phi^\bsy \leq
  \frac{\alpha_1+\alpha_2\,\|E_T\|_{\calX\to L^2(D)}^2}{2}\,
 \Big(\frac{\|f\|_{\calY'}}{\beta_1}+\|\widehat{u}\|_{\calX}\Big)^2 < \infty.
\end{align}
Thus for all $\theta>0$ we have
\begin{align}
 &1 \le \exp\big(\theta\,\Phi^\bsy\big) \le e^\sigma <\infty, \quad\mbox{with}\quad \label{eq:sigma1} \\
 &\sigma := \frac{\alpha_1+\alpha_2\,\|E_T\|_{\calX\to L^2(D)}^2}{2}\,
 \Big(\frac{\|f\|_{\calY'}}{\beta_1}+\|\widehat{u}\|_{\calX}\Big)^2 \theta.\label{eq:sigma2}
\end{align}
\end{lemma}

\begin{proof}
We have from \eqref{eq:Phi} that
\begin{align*}
 \Phi^\bsy(z)
 &\le \frac{\alpha_1}{2} \big\|(B^\bsy)^{-1}f- \widehat{u}\big\|^2_{\calX}
 + \frac{\alpha_2}{2} \|E_T\|_{\calX\to L^2(D)}^2 \big\|(B^\bsy)^{-1}f-\widehat{u}\big\|^2_{\calX} \nonumber\\
 &\le \frac{\alpha_1+\alpha_2\,\|E_T\|_{\calX\to L^2(D)}^2}{2}\,
 \big(\big\|(B^\bsy)^{-1}f\big\|_\calX + \|\widehat{u}\|_\calX\big)^2\,,
\end{align*}
which yields \eqref{eq:Phi-bound} after applying \eqref{eq:u-apriori}.
 
\end{proof}

Using the preceding lemma, we compute the gradient of \eqref{eq:J-entropic}.
\begin{lemma}
Let $\alpha_1,\alpha_2 \geq 0$ and $\alpha_3 > 0$, with $\alpha_1 +
\alpha_2 > 0$, and let $0<\theta<\infty$. Let $f = (z,u_0) \in \calY'$ and
$\widehat{u} \in \calX$. For every $\bsy \in U$, let $u^\bsy \in
\calX$ be the solution of \eqref{eq:model} and then let $q^\bsy =
(q_1^\bsy,q_2^\bsy) \in \calY$ be the solution of \eqref{eq:weakdual} with
$f_{\rm dual}$ given by \eqref{eq:f-dual2}. Then the gradient of
\eqref{eq:J-entropic} is given as an element of $L^2(V;I)$ for $z \in
L^2(V';I)$ by
\begin{align}\label{eq:gradient_entropic}
    J'(z) = \frac{1}{\int_U  \exp\big(\theta\,\Phi^\bsy(z)\big) \,\mathrm d\bsy}
    \int_U \exp\big(\theta\,\Phi^\bsy(z)\big)\, q_1^\bsy\,\mathrm d\bsy + \alpha_3 R_V^{-1}z
\end{align}
where $\Phi^\bsy(z)$ is defined in \eqref{eq:Phi}.
\end{lemma}

\begin{proof}
The application of the chain rule gives
\begin{align*}
    \partial_z \mathcal R_{{\rm e}}(\Phi^\bsy(z)) = \frac{1}{\theta\int_U  \exp\big(\theta\,\Phi^\bsy(z)\big)\, \mathrm d\bsy}
    \partial_z \Big(\int_U \exp\big(\theta\,\Phi^\bsy(z)\big)\,\mathrm d\bsy\Big)\,.
\end{align*}
Lemma \ref{lemma:Phi} implies that $1 \le \int_U
\exp\big(\theta\,\Phi^\bsy(z)\big)\,\rd\bsy <\infty$. Then the integral
is a bounded and linear operator and hence its Fr\'echet derivative is the
operator itself. Exploiting this fact, we obtain
that $\partial_z \left(\int_U \exp\big(\theta\,\Phi^\bsy(z)\big)\,\mathrm
d\bsy\right) = \int_U \left(
\partial_z \exp\big(\theta\,\Phi^\bsy(z)\big)\right) \mathrm d\bsy$. By
the chain rule it follows for each $\bsy \in U$ that $
 \partial_z
\exp\big(\theta\,\Phi^\bsy(z)\big) = \theta \,
\exp\big(\theta\,\Phi^\bsy(z)\big)\, \partial_z \Phi^\bsy(z)\,.
$
Recalling
from the previous subsection that $\partial_z
(\frac{\alpha_3}{2}\|z\|_{L^2(V';I)}^2) = \alpha_3 R_V^{-1} z$ and
$
    \partial_z \Phi^\bsy(z) = \big((B_1^\bsy)^{\dagger}\big)'(\alpha_1 R_V +\alpha_2 E_T'E_T)(u^\bsy(z)-\widehat{u}) = q_1^\bsy\,,
$
and collecting terms gives \eqref{eq:gradient_entropic}.  
\end{proof}

\subsection{Optimality conditions} \label{sec:opt-cond}

In the case when the feasible set of controls $\mathcal{Z}$ is a nonempty
and convex set, we know (see, e.g., \cite[Lemma
2.21]{Troeltzsch2010}) that the optimal control $z^*$ satisfies the
variational inequality
\begin{align} \label{eq:var-ineq}
  \langle J'(z^*), z - z^* \rangle_{L^2(V;I),L^2(V';I)} \geq 0 \quad \forall z \in \mathcal Z\,.
\end{align}
For convex objective functionals $J(z)$, like the ones considered in this
work, the variational inequality is a necessary and sufficient condition
for optimality. The complete optimality conditions are then given by the
following result.
\begin{theorem} \label{thm:KKT}
Let $\mathcal R$ be the expected value or the entropic risk measure. A
control $z^* \in L^2(V';I)$ is the unique minimizer of
\eqref{eq:objective} subject to \eqref{eq:model} and
\eqref{eq:boxconstraints} if and only if it satisfies the optimality
system: 
\begin{align*}
 \begin{cases}
 \langle B^\bsy u^\bsy, (v_1,v_2) \rangle_{\calY',\calY}
 = \langle z^*, v_1 \rangle_{L^2(V';I),L^2(V;I)} + \langle u_0, v_2 \rangle_{L^2(D)} \,\,\,\forall\,v\in \calY,
 \\
 \langle w, (B^\bsy)' q^\bsy \rangle_{\calX,\calX'}
 = \langle w, \alpha_1 R_V(u^\bsy - \widehat{u}) \rangle_{\calX,\calX'} \\
 \qquad\qquad\qquad\qquad\qquad
 + \langle w(T), \alpha_2(u^\bsy(T)-\widehat{u}(T)) \rangle_{L^2(D)} \,\,\, \forall\, w \in \calX,
 \\
 z^* \in \mathcal{Z}\,,
 \\
 \langle J'(z^*), z - z^* \rangle_{L^2(V;I),L^2(V';I)} \geq 0 \quad \forall z \in \mathcal Z\,,
 \end{cases}
\end{align*}
which holds for all $\bsy \in U$, and $J'(z)$ is given by
\eqref{eq:gradient} for the expected value, or
\eqref{eq:gradient_entropic} for the entropic risk measure.
\end{theorem}

Observe that the optimality system in Theorem~\ref{thm:KKT} contains
the variational formulations of the state PDE \eqref{eq:weakPDE} and the
dual PDE \eqref{eq:weakdual} in the first and second equation,
respectively.

It is convenient to reformulate the variational inequality
\eqref{eq:var-ineq} in terms of an orthogonal projection onto $\mathcal
Z$. The orthogonal projection onto a nonempty, closed and convex subset
$\mathcal{Z} \subset H$ of a Hilbert space $H$, denoted by
$P_{\mathcal{Z}}: H \to \mathcal{Z}$, is defined as
\begin{align*}
    P_{\mathcal{Z}}(h) \in \mathcal{Z}\,, \quad \|P_{\mathcal{Z}}(h) - h\|_H = \min_{v \in \mathcal{Z}} \|v - h\|_H\,, \quad \forall h \in H\,.
\end{align*}
Then, see, e.g., \cite[Lemma 1.11]{hinze2009optimization}, for all $h \in
H$ and $\gamma >0$ the condition $h \in \mathcal{Z}$, $\langle h, v -
z\rangle_{H} \geq 0\, \forall v \in \mathcal{Z}$ is equivalent to $z -
P_{\mathcal{Z}}(z - \gamma h) = 0$. Using the definition of the Riesz
operator and $H = L^2(V';I)$, we conclude that \eqref{eq:var-ineq} is
equivalent to
\begin{align*}
    z^* - P_{\mathcal{Z}}(z^* - \gamma R_V J'(z^*)) = 0\,.
\end{align*}
This equivalence can then be used to develop projected descent methods to solve the optimal control problem, see, e.g., \cite[Chapter 2.2.2]{hinze2009optimization}.
\begin{remark}\label{rem:remark}
If $\mathcal{Z}$ is the closed ball with radius $r > 0$ in a Hilbert space
$H$, then the orthogonal projection $P_{\mathcal{Z}}$ is given by
\begin{align*}
   P_{\mathcal{Z}}(h) = \min{\Big(1,\frac{r}{\|h\|_H}\Big)}\,h
    \qquad\mbox{for all } h \in H.
\end{align*}
\end{remark}

\section{Parametric regularity of the adjoint state} \label{sec:reg-adj}

In this section we derive an a priori bound for the adjoint state and the
partial derivatives of the adjoint state with respect to the parametric
variables. Existing results, e.g., \cite[Theorem 4]{KunothSchwab}, do
not directly apply to our case, since the right-hand side of the affine
linear, parametric operator equation depends on the parametric variable,
more specifically
\[
 (B^\bsy)' q^\bsy = (\alpha_1 R_V + \alpha_2 E_T'E_T)(u^\bsy - \widehat{u}).
\]

\begin{lemma} \label{lem:q-apriori} Let $\alpha_1,\alpha_2 \geq 0$ and
$\alpha_3 > 0$, with $\alpha_1 + \alpha_2 > 0$. Let $f = (z,u_0) \in
\calY'$ and $\widehat{u} \in \calX$. For every $\bsy \in U$, let
$u^\bsy \in \calX$ be the solution of \eqref{eq:model} and then let
$q^\bsy \in \calY$ be the solution of \eqref{eq:weakdual} with $f_{\rm
dual}$ given by \eqref{eq:f-dual2}. Then we have
\begin{align*}
  \|q^\bsy\|_{\calY}
  \le \frac{\alpha_1 + \alpha_2\,\|E_T\|^{2}_{\calX\to L^2(D)}}{\beta_1}
  \bigg(\frac{\|f\|_{\calY'}}{\beta_1} + \|\widehat{u}\|_\calX\bigg),
\end{align*}
where $\beta_1$ is described in \eqref{eq:beta}.
\end{lemma}

\begin{proof}
By the bounded invertibility of $B^\bsy$ and its dual operator, we have
\begin{align*}
 \|q^\bsy\|_\calY
 &\le \|((B^\bsy)')^{-1}\|_{\calX'\to \calY}\,
 \|(\alpha_1 R_V + \alpha_2 E_T'E_T)(u^\bsy - \widehat{u})\|_{\calX'}
\end{align*}
with $\|((B^\bsy)')^{-1}\|_{\calX'\to \calY} \le 1/\beta_1$,
\begin{align*}
 \|R_V(u^\bsy-\widehat{u})\|_{\calX'}
 &\le \|R_V(u^\bsy-\widehat{u})\|_{L^2(V';I)}  = \|u^\bsy-\widehat{u}\|_{L^2(V;I)} \le \|u^\bsy-\widehat{u}\|_{\calX},
 \\
 \|E_T'E_T(u^\bsy - \widehat{u})\|_{\calX'}
 & \leq \|E_T\|^{2}_{\calX\to L^2(D)}\, \|u^\bsy - \widehat{u}\|_{\calX},
 \\
 \|u^\bsy - \widehat{u}\|_{\calX}
 &\le \|u^\bsy\|_\calX + \|\widehat{u}\|_\calX
 \le \frac{\|f\|_{\calY'}}{\beta_1} + \|\widehat{u}\|_\calX,
\end{align*}
where we used \eqref{eq:u-apriori}. Combining the estimates gives the
desired result.  
\end{proof}

\begin{theorem}\label{thm:adjregularity}
Let $\alpha_1,\alpha_2 \geq 0$ and $\alpha_3 > 0$, with $\alpha_1 +
\alpha_2 > 0$. Let $f = (z,u_0) \in \calY'$ and $\widehat{u} \in
\calX$. For every $\bsy \in U$, let $u^\bsy \in \calX$ be the solution of
\eqref{eq:model} and then let $q^\bsy \in \calY$ be the solution of
\eqref{eq:weakdual} with $f_{\rm dual}$ given by \eqref{eq:f-dual2}. Then
for every $\bsnu\in\mathscr{F}$ we have
\begin{align*}
  \|\partial^\bsnu_\bsy q^\bsy\|_{\calY}
  \le \frac{\alpha_1 + \alpha_2\,\|E_T\|^{2}_{\calX\to L^2(D)}}{\beta_1}
  \Big(\frac{\|f\|_{\calY'}}{\beta_1} + \|\widehat{u}\|_\calX\Big)\,
  (|\bsnu| + 1)!\,\bsb^\bsnu,
\end{align*}
where $\beta_1$ is described in \eqref{eq:beta} and the sequence $\bsb =
(b_j)_{j\ge 1}$ is defined in \eqref{eq:bj}.
\end{theorem}

\begin{proof}
For $\bsnu = \bszero$ the assertion follows from the previous lemma. For
$\bsnu \neq \bszero$ we take derivatives $\partial_\bsy^\bsnu ((B^\bsy)'
q^\bsy) = \partial_\bsy^\bsnu ((\alpha_1 R_V + \alpha_2 E_T'E_T)(u^\bsy -
\widehat{u}))$ and use the Leibniz product rule to get
\begin{align*}
  \sum_{\bsm \leq \bsnu} \binom{\bsnu}{\bsm}
  \big(\partial_\bsy^{\bsm} (B^\bsy)'\big) \big(\partial_\bsy^{\bsnu-\bsm} q^\bsy\big)
  = (\alpha_1 R_V + \alpha_2 E_T'E_T)\big(\partial_\bsy^\bsnu (u^\bsy - \widehat{u})\big)\,.
\end{align*}
Separating out the $\bsm = \bszero$ term, we obtain
\begin{align*}
  &(B^\bsy)' (\partial_\bsy^\bsnu q^\bsy) \\
  &= - \sum_{\bszero\ne\bsm \leq \bsnu} \binom{\bsnu}{\bsm}
  \big(\partial_\bsy^{\bsm} (B^\bsy)'\big) \big(\partial_\bsy^{\bsnu-\bsm} q^\bsy\big)
  + (\alpha_1 R_V + \alpha_2 E_T'E_T)\big(\partial_\bsy^\bsnu (u^\bsy - \widehat{u})\big).
\end{align*}
By the bounded invertibility of $(B^\bsy)'$, we have
$\|((B^\bsy)')^{-1}\|_{\calX'\to \calY} \leq \tfrac{1}{\beta_1}$ and
\begin{align*}
  \|\partial_\bsy^\bsnu q^\bsy\|_{\calY}
  &\leq \sum_{\bszero\ne\bsm \leq \bsnu} \binom{\bsnu}{\bsm}
  \|((B^\bsy)')^{-1} \partial_\bsy^{\bsm} (B^\bsy)' \|_{\calY\to \calY}\,
  \|\partial_\bsy^{\bsnu-\bsm} q^\bsy\|_{\calY}\\
  &\qquad + \|((B^\bsy)')^{-1}\|_{\calX' \to \calY}\,
  \|(\alpha_1 R_V + \alpha_2 E_T'E_T)(\partial_\bsy^\bsnu (u^\bsy - \widehat{u}))\|_{\calX'} \\
  &\leq \sum_{\bszero\ne\bsm \leq \bsnu} \binom{\bsnu}{\bsm} \frac{1}{\beta_1}
  \|\partial_\bsy^{\bsm} (B^\bsy)' \|_{\calY\to \calX'}\, \|\partial_\bsy^{\bsnu-\bsm} q^\bsy\|_{\calY}  \\
  &\qquad+ \frac{\alpha_1 + \alpha_2\,\|E_T\|^{2}_{\calX\to L^2(D)}}{\beta_1}
  \|\partial_\bsy^\bsnu (u^\bsy-\widehat{u})\|_{\calX}.
\end{align*}
Recall that
\begin{align*}
 &\langle v, (B^\bsy)' w\rangle_{\calX,\calX'} \\
 &= \int_I \langle v, -\tfrac{\partial}{\partial t} w\rangle_{V,V'}\,\mathrm dt + \int_I \int_D a^\bsy\, \nabla v \cdot \nabla w\, \mathrm dx\, \mathrm dt
  + \int_D E_Tw\, E_Tv\, \rd x.
\end{align*}
For $\bsm\ne\bszero$, we conclude with \eqref{eq:axy} that $\langle
v,\partial^{\bsm} (B^\bsy)' w\rangle_{\calX,\calX'} = \int_I \int_D
\psi_j\, \nabla v \cdot \nabla w \,\rd x \,\rd t$ if
$\bsm=\boldsymbol{e}_j$, and otherwise it is zero. Hence for $\bsm
=\bse_j$ we obtain for all $v\in \calY$ that
\begin{align*}
\|\partial^{\bsm}(B^\bsy)' v\|_{\calX'}
&= \sup_{w\in\calX} \frac{|\langle v,\partial^{\bsm} (B^\bsy)' w\rangle_{\calX,\calX'}|}{\|w\|_\calX}
= \sup_{w\in\calX} \frac{|\int_I \int_D \psi_j\, \nabla v \cdot \nabla w \,\mathrm dx\, \mathrm dt |}{\|w\|_\calX}\\
 &\leq b_j\,\sup_{w\in\calX} \frac{\| v\|_{L^2(V;I)}\, \|w \|_{L^2(V;I)}}{\|w\|_\calX}
 \leq b_j \|v\|_{\calY}\,.
\end{align*}
Hence
\begin{align*}
 \|\partial_\bsy^\bsnu q^\bsy\|_\calY \leq \!\!\!\sum_{j \in \supp(\bsnu)}\!\!\!
 \nu_j\, b_j\, \|\partial_\bsy^{\bsnu-\boldsymbol{e}_j}q^\bsy\|_\calY
 + \frac{\alpha_1 + \alpha_2\,\|E_T\|^{2}_{\calX\to L^2(D)}}{\beta_1}
 \|\partial_\bsy^\bsnu (u^\bsy-\widehat{u})\|_{\calX}.
\end{align*}
By Lemma~\ref{lem:q-apriori} this recursion is true for $\bsnu = \bszero$
and we may apply \cite[Lemma 9.1]{kuonuyenssurvey} to get
\begin{align*}
 \|\partial_\bsy^\bsnu q^\bsy\|_\calY \leq \sum_{\bsm \leq \bsnu} \binom{\bsnu}{\bsm}\,|\bsm|!\,\bsb^\bsm
 \Big(\frac{\alpha_1 + \alpha_2\,\|E_T\|^{2}_{\calX\to L^2(D)}}{\beta_1}
 \|\partial_\bsy^{\bsnu-\bsm} (u^\bsy-\widehat{u})\|_{\calX}\Big)\,.
\end{align*}

From \eqref{eq:u-apriori} and \eqref{eq:statebound} we have
$$
 \|\partial^{\bsnu}_\bsy(u^\bsy-\widehat{u})\|_{\calX}
 \le \begin{cases}
 \frac{1}{\beta_1}\|f\|_{\calY'} + \|\widehat{u}\|_{\calX} & \text{if}~\bsnu=\bszero, \vspace{0.1cm} \\
 \frac{1}{\beta_1}\|f\|_{\calY'}\,|\bsnu|!\,\bsb^{\bsnu} &\text{if}~\bsnu\ne\bszero.
  \end{cases}
$$
We finally arrive at
\begin{align*}
 \|\partial_\bsy^\bsnu q^\bsy\|_{\calY}
 &\leq \sum_{\satop{\bsm \leq \bsnu}{\bsm\ne\bsnu}} \binom{\bsnu}{\bsm}\,|\bsm|!\,\bsb^\bsm
 \frac{\alpha_1 + \alpha_2\,\|E_T\|^{2}_{\calX\to L^2(D)}}{\beta_1}\,
 \frac{\|f\|_{\calY'}}{\beta_1} \,|\bsnu-\bsm|!\,\bsb^{\bsnu-\bsm} \\
 &\qquad + |\bsnu|!\,\bsb^\bsnu\,\frac{\alpha_1 + \alpha_2\,\|E_T\|^{2}_{\calX\to L^2(D)}}{\beta_1}
 \Big(\frac{\|f\|_{\calY'}}{\beta_1} + \|\widehat{u}\|_{\calX}\Big) \\
 &= (|\bsnu| +1)!\, \bsb^\bsnu
 \frac{\alpha_1 + \alpha_2\,\|E_T\|^{2}_{\calX\to L^2(D)}}{\beta_1}\,
 \frac{\|f\|_{\calY'}}{\beta_1}\\
 &\qquad+ |\bsnu|!\, \bsb^\bsnu \frac{\alpha_1 + \alpha_2\,\|E_T\|^{2}_{\calX\to L^2(D)}}{\beta_1} \|\widehat{u}\|_{\calX} \\
 &\le (|\bsnu| +1)!\, \bsb^\bsnu
 \frac{\alpha_1 + \alpha_2\,\|E_T\|^{2}_{\calX\to L^2(D)}}{\beta_1}\,
 \Big(\frac{\|f\|_{\calY'}}{\beta_1} + \|\widehat{u}\|_{\calX}\Big),
\end{align*}
where the equality follows from \cite[Formula~(9.4)]{kuonuyenssurvey}.
 
\end{proof}

\section{Regularity analysis for the entropic risk measure}
\label{sec:reg-ent}

Our goal is to use QMC to approximate the following high-dimensional
integrals appearing in the denominator and numerator of the
gradient~\eqref{eq:gradient_entropic}. To this end, we develop regularity
bounds for the integrands.

\begin{lemma} \label{lem:Psibound}
Let $\theta>0$, $\alpha_1,\alpha_2 \geq 0$, with $\alpha_1 + \alpha_2 >
0$. Let $f = (z,u_0) \in \calY'$ and $\widehat{u} \in \calX$. For
every $\bsy \in U$, let $u^\bsy \in \calX$ be the solution of
\eqref{eq:model} and let $\Phi^\bsy$ be as in~\eqref{eq:Phi}. Then for all
$\bsnu\in\mathscr{F}$ we have
\begin{align*}
 |\partial^{\bsnu}_\bsy \Phi^\bsy|
 \le \frac{\alpha_1+\alpha_2\,\|E_T\|_{\calX\to L^2(D)}^2}{2}\,
 \bigg(\frac{\|f\|_{\calY'}}{\beta_1}+\|\widehat{u}\|_{\calX}\bigg)^2\,
 (|\bsnu|+1)!\,\bsb^{\bsnu},
\end{align*}
where the sequence $\bsb = (b_j)_{j\ge 1}$ is defined
by~\eqref{eq:bj}.
\end{lemma}

\begin{proof}
The case $\bsnu=\bszero$ is precisely \eqref{eq:Phi-bound}. Consider now
$\bsnu\ne\bszero$. We estimate the partial derivatives of $\Phi^\bsy$ by
differentiating under the integral sign and using the Leibniz product rule
in conjunction with the Cauchy--Schwarz inequality to obtain
$$
 |\partial^{\bsnu}_\bsy \Phi^\bsy|\leq
 \frac{\alpha_1+\alpha_2\,\|E_T\|_{\calX\to L^2(D)}^2}{2}
 \sum_{\bsm\leq\bsnu}\binom{\bsnu}{\bsm}\|\partial^{\bsm}(u^\bsy-\widehat{u})\|_{\calX}\,
 \|\partial^{\bsnu-\bsm}(u^\bsy-\widehat{u})\|_{\calX}.
$$
Separating out the $\bsm=\bszero$ and $\bsm=\bsnu$ terms and
utilizing~\eqref{eq:statebound}, we obtain
\begin{align*}
&\sum_{\bsm\leq\bsnu}\binom{\bsnu}{\bsm}\, \|\partial^{\bsm}(u^\bsy-\widehat{u})\|_{\calX}\,
 \|\partial^{\bsnu-\bsm}(u^\bsy-\widehat{u})\|_{\calX}\\
&= 2\,\|u^\bsy-\widehat{u}\|_{\calX}\, \|\partial^{\bsnu}u^\bsy\|_{\calX}
 +\sum_{\substack{\bsm\leq \bsnu\\ \bszero\neq \bsm\neq \bsnu}}
 \binom{\bsnu}{\bsm}\, \|\partial^{\bsm}u^\bsy\|_{\calX}\, \|\partial^{\bsnu-\bsm}u^\bsy\|_{\calX}\\
&\le 2\,\bigg(\frac{\|f\|_{\calY'}}{\beta_1}+\|\widehat{u}\|_{\calX}\bigg) \frac{\|f\|_{\calY'}}{\beta_1}
     |\bsnu|!\,\bsb^{\bsnu}
    + \bigg(\frac{\|f\|_{\calY'}}{\beta_1}\bigg)^2 \bsb^{\bsnu} \!\!
    \sum_{\substack{\bsm\leq \bsnu\\ \bszero\neq \bsm\neq \bsnu}}\!\binom{\bsnu}{\bsm}\,|\bsm|!\,|\bsnu-\bsm|!,
\end{align*}
where the sum over $\bsm$ can be rewritten as
\begin{align*}
 \sum_{\ell=1}^{|\bsnu|-1}\ell!\,(|\bsnu|-\ell)!\sum_{\bsm\le\bsnu,\,|\bsm|=\ell} \binom{\bsnu}{\bsm}
 = \sum_{\ell=1}^{|\bsnu|-1}\ell!\,(|\bsnu|-\ell)!\,\binom{|\bsnu|}{\ell}
 = |\bsnu|!\,(|\bsnu|-1),
\end{align*}
where we used the identity
\begin{align} \label{eq:vander}
 \sum_{\bsm\le\bsnu,\,|\bsm|=\ell} \binom{\bsnu}{\bsm}
 = \binom{|\bsnu|}{\ell} = \frac{|\bsnu|!}{(|\bsnu|-\ell)!\,\ell!},
\end{align}
which is a simple consequence of the Vandermonde
convolution~\cite[Equation~(5.1)]{gouldbook}. Combining the estimates
yields the required result.  
\end{proof}

For future reference, we state a recursive form of Fa\`a di Bruno's
formula~\cite{savits} for the exponential function.

\begin{theorem} \label{thm:faadibruno}
Let $G : U\to \bbR$. For all $\bsy\in U$ and
$\bsnu\in\mathscr{F}\setminus\{\bszero\}$, we have
\begin{align*} 
 \partial^{\bsnu}_\bsy \exp(G(\bsy)) = \exp(G(\bsy))
 \sum_{\lambda=1}^{|\bsnu|} \alpha_{\bsnu,\lambda}(\bsy),
\end{align*}
where the sequence $(\alpha_{\bsnu,\lambda}(\bsy))_{\bsnu\in\mathscr
F,\lambda\in \mathbb N_0}$ is defined recursively by $\alpha_{\bsnu,
0}(\bsy)=\delta_{\bsnu,\bszero}$, $\alpha_{\bsnu,\lambda}(\bsy)=0$ for
$\lambda>|\bsnu|$, and otherwise
\begin{align*}
 \alpha_{\bsnu+\boldsymbol e_j,\lambda}(\bsy)
  = \sum_{\bsm\le\bsnu}\binom{\bsnu}{\bsm}\, (\partial^{\bsnu-\bsm+\boldsymbol e_j}G)(\bsy)\,
  \alpha_{\bsm,\lambda-1}(\bsy), \qquad j\ge 1.
\end{align*}
\end{theorem}

\begin{proof}
This is a special case of \cite[Formulas~(3.1) and (3.5)]{savits} in which
$f$ is the exponential function and $m=1$ so that $\lambda$ is an integer.
 
\end{proof}

\begin{lemma} \label{lem:recur}
Let the sequence
$(\bbA_{\bsnu,\lambda})_{\bsnu\in\mathscr{F},\,\lambda\in\bbN_0}$ satisfy
$\bbA_{\bsnu,0}=\delta_{\bsnu,\bszero}$, $\bbA_{\bsnu,\lambda}=0$ for
$\lambda>|\bsnu|$, and otherwise satisfy the recursion
\begin{align} \label{eq:recur}
  \bbA_{\bsnu+\boldsymbol e_j,\lambda}
  \le \sum_{\bsm\le\bsnu}\binom{\bsnu}{\bsm}\, c\, \bsrho^{\bsnu-\bsm+\bse_j}\,(|\bsnu|-|\bsm|+2)!\,
   \bbA_{\bsm,\lambda-1}, \qquad j\ge 1,
\end{align}
for some $c>0$ and a nonnegative sequence $\bsrho$. Then for all
$\bsnu\ne\bszero$ and $1\leq \lambda\leq |\bsnu|$ we have
\begin{align} \label{eq:goal}
  \bbA_{\bsnu,\lambda} \,\le\, c^\lambda\, \bsrho^\bsnu
  \sum_{k=1}^\lambda \frac{(-1)^{\lambda+k}\,(|\bsnu|+2k-1)!}{(2k-1)!\,(\lambda-k)!\,k!}.
\end{align}
The result is sharp in the sense that both inequalities can be replaced by
equalities.
\end{lemma}

\begin{proof}
We prove \eqref{eq:goal} for all $\bsnu\ne\bszero$ and $1\leq
\lambda\leq |\bsnu|$ by induction on $|\bsnu|$. The base case
$\bbA_{\bse_j,1}$ is easy to verify. Let $\bsnu\ne\bszero$ and suppose
that \eqref{eq:goal} holds for all multi-indices $\bsm$ of order $\le
|\bsnu|$ and all $1\leq \lambda\leq |\bsm|$. The case
$\bbA_{\bsnu+\bse_j,1}$ is also straightforward to verify. We consider
therefore $2\leq \lambda \leq |\bsnu|+1$. Using \eqref{eq:recur} and the
induction hypothesis, we have
\begin{align} \label{eq:T}
  &\bbA_{\bsnu+\boldsymbol e_j,\lambda} \nonumber \\
  &\le \sum_{\bszero\ne\bsm\le\bsnu}\binom{\bsnu}{\bsm}c\,
  \bsrho^{\bsnu-\bsm+\bse_j}(|\bsnu|-|\bsm|+2)!\, \nonumber \\
  &\qquad  \times \bigg(c^{\lambda-1}\,
  \bsrho^\bsm \sum_{k=1}^{\lambda-1} \frac{(-1)^{\lambda-1+k}\,(|\bsm|+2k-1)!}{(2k-1)!\,(\lambda-1-k)!\,k!}\bigg) \nonumber\\
  &= c^\lambda\, \bsrho^{\bsnu+\bse_j} \sum_{\ell=1}^{|\bsnu|} \sum_{\satop{\bsm\le\bsnu}{|\bsm|=\ell}} \binom{\bsnu}{\bsm}\,
  \sum_{k=1}^{\lambda-1} \frac{(-1)^{\lambda-1+k}\, (|\bsnu|-\ell+2)!\,(\ell+2k-1)!}{(2k-1)!\,(\lambda-1-k)!\,k!}
  \nonumber\\
  &= c^\lambda\, \bsrho^{\bsnu+\bse_j}\,
  \frac{2\,|\bsnu|!\,(-1)^{\lambda-1}}{(\lambda-1)!}
  \underbrace{\sum_{k=1}^{\lambda-1} (-1)^k\,\binom{\lambda-1}{k}
  \sum_{\ell=1}^{|\bsnu|} \binom{|\bsnu|-\ell+2}{|\bsnu|-\ell} \binom{\ell+2k-1}{\ell}}_{=:\,T},
\end{align}
where we used \eqref{eq:vander} and then regrouped the factors as
binomial coefficients. Next we take the binomial identity
\cite[Equation~(5.6)]{gouldbook}
\[
  \sum_{\ell=0}^{|\bsnu|} \binom{|\bsnu|-\ell+2}{|\bsnu|-\ell} \binom{\ell+2k-1}{\ell}
  = \binom{|\bsnu|+2k+2}{|\bsnu|},
\]
separate out the $\ell=0$ term, and use $\sum_{k=1}^{\lambda-1} (-1)^k
\binom{\lambda-1}{k} = \sum_{k=0}^{\lambda-1} (-1)^k \binom{\lambda-1}{k}
-1 = -1$, to rewrite $T$ as
\begin{align*}
  T &= \sum_{k=1}^{\lambda-1} (-1)^k\,\binom{\lambda-1}{k}
  \bigg[ \binom{|\bsnu|+2k+2}{|\bsnu|} - \binom{|\bsnu|+2}{|\bsnu|} \bigg] \\
 &= \sum_{k=1}^{\lambda-1} (-1)^k\,\binom{\lambda-1}{k}
  \binom{|\bsnu|+2k+2}{|\bsnu|} + \binom{|\bsnu|+2}{|\bsnu|} \\
 &= \sum_{k=0}^{\lambda-1} (-1)^k\,\binom{\lambda-1}{k}
  \binom{|\bsnu|+2k+2}{|\bsnu|}
 = \sum_{k=1}^{\lambda} (-1)^{k-1}\,\binom{\lambda-1}{k-1}
  \binom{|\bsnu|+2k}{|\bsnu|}.
\end{align*}
Substituting this back into \eqref{eq:T} and simplifying the factors, we
obtain
\[
  \bbA_{\bsnu+\boldsymbol e_j,\lambda}
  \le c^\lambda\, \bsrho^{\bsnu+\bse_j} \sum_{k=1}^{\lambda}
   \frac{(-1)^{\lambda+k}\,(|\bsnu|+2k)!}{(2k-1)!\,(\lambda-k)!\,k!},
\]
as required.  
\end{proof}

\begin{theorem} \label{thm:expbound}
Let $\theta>0$, $\alpha_1,\alpha_2 \geq 0$, with $\alpha_1 + \alpha_2 >
0$. Let $f = (z,u_0) \in \calY'$ and $\widehat{u} \in \calX$. For
every $\bsy \in U$, let $u^\bsy \in \calX$ be the solution of
\eqref{eq:model} and let $\Phi^\bsy$ be as in~\eqref{eq:Phi}. Then for all
$\bsnu\in\mathscr{F}$ we have
$$
 |\partial^{\bsnu}_\bsy \exp(\theta\,\Phi^\bsy)|
 \le e^{\max(\sigma,\,\sigma e^2+2\sigma-1)}\, |\bsnu|!\,(e\bsb)^{\bsnu},
$$%
where the sequence $\bsb = (b_j)_{j\ge 1}$ is defined by~\eqref{eq:bj} and
$\sigma$ is defined by \eqref{eq:sigma2}.
\end{theorem}

\begin{proof}
For $\bsnu=\bszero$ we have from \eqref{eq:sigma1} that
$|\exp(\theta\,\Phi^\bsy)| \le e^\sigma$, which satisfies the required
bound. For $\bsnu\ne\bszero$, from Fa\`a di Bruno's formula
(Theorem~\ref{thm:faadibruno}) we have
\begin{align} \label{eq:step1}
 |\partial^{\bsnu}_\bsy \exp(\theta\Phi^\bsy)|
 \le \exp(\theta\,\Phi^\bsy) \sum_{\lambda=1}^{|\bsnu|} |\alpha_{\bsnu,\lambda}(\bsy)|,
\end{align}
with $\alpha_{\bsnu, 0}(\bsy)=\delta_{\bsnu,\bszero}$,
$\alpha_{\bsnu,\lambda}(\bsy)=0$ for $\lambda>|\bsnu|$, and
\begin{align*}
  |\alpha_{\bsnu+\bse_j,\lambda}(\bsy)|
  &\le \sum_{\bsm\le\bsnu}\binom{\bsnu}{\bsm}\, \theta\,|\partial^{\bsm+\bse_j}_\bsy \Phi^\bsy|\,
   |\alpha_{\bsnu-\bsm,\lambda-1}(\bsy)| \\
  &\le \sum_{\bsm\le\bsnu}\binom{\bsnu}{\bsm}\, \sigma \,(|\bsm|+2)!\,\bsb^{\bsm+\bse_j}\,
   |\alpha_{\bsnu-\bsm,\lambda-1}(\bsy)|,
\end{align*}
where we used Lemma~\ref{lem:Psibound}. Applying Lemma~\ref{lem:recur} we
conclude that
\begin{align} \label{eq:step2}
  |\alpha_{\bsnu,\lambda}(\bsy)| \le \sigma^\lambda\,\bsb^\bsnu
  \sum_{k=1}^\lambda \frac{(-1)^{\lambda+k}\,(|\bsnu|+2k-1)!}{(2k-1)!\,(\lambda-k)!\,k!}.
\end{align}
We have
\begin{align} \label{eq:step3}
 &\sum_{\lambda=1}^{|\bsnu|} \sigma^\lambda
 \sum_{k=1}^\lambda \frac{(-1)^{\lambda+k}\,(|\bsnu|+2k-1)!}{(2k-1)!\,(\lambda-k)!\,k!}
 = \sum_{k=1}^{|\bsnu|} \frac{(|\bsnu|+2k-1)!}{(2k-1)!\,k!}
 \sum_{\lambda=k}^{|\bsnu|} \frac{(-1)^{\lambda+k}\,\sigma^\lambda}{(\lambda-k)!} \nonumber\\
 &= |\bsnu|!\sum_{k=1}^{|\bsnu|} \frac{\sigma^k}{k!} \binom{|\bsnu|+2k-1}{2k-1}
 \sum_{\ell=0}^{|\bsnu|-k} \frac{(-\sigma)^\ell}{\ell!}
 \le |\bsnu|! \sum_{k=1}^{|\bsnu|} \frac{\sigma^k}{k!} e^{|\bsnu|+2k-1} e^\sigma\\
 &\le |\bsnu|!\, e^{|\bsnu|+\sigma e^2+\sigma-1}, \nonumber
\end{align}
where we used $\binom{n}{m} \le n^m/m! \le e^n$. Combining
\eqref{eq:step1}, \eqref{eq:step2}, \eqref{eq:step3} and
\eqref{eq:Phi-bound} gives
\[
|\partial^{\bsnu}_\bsy \exp(\theta\Phi^\bsy)|
 \le \exp(\sigma)\,\bsb^\bsnu\,|\bsnu|!\, e^{|\bsnu|+\sigma e^2+\sigma-1}
 = e^{\sigma e^2+2\sigma-1}\, |\bsnu|!\,(e\bsb)^\bsnu,
\]
as required.  
\end{proof}

\begin{remark}
In the proof of Theorem~\ref{thm:expbound}, a different manipulation of
\eqref{eq:step3} can yield a different bound $2c\, e^{|\bsnu|+\sigma
e^2+\sigma+1}(|\bsnu|-1)!$ for $\bsnu\ne\bszero$, leading to a tighter
upper bound for large $|\bsnu|$ at the expense of a bigger constant,
$$
 |\partial^{\bsnu}_\bsy \exp(\theta\,\Phi^\bsy)|
 \le 2\sigma\,{\rm e}^{\sigma e^2+2\sigma+1}\,(|\bsnu|-1)!\,(e\bsb)^{\bsnu}.
$$
This then leads to a more complicated bound for
Theorem~\ref{thm:expbound-q1} below. We have chosen to present the current
form of Theorem~\ref{thm:expbound} to simplify our subsequent analysis.

Interestingly, the sum in \eqref{eq:step2} can also be rewritten as a sum
with only positive terms: denoting $v = |\bsnu|$,
\begin{align*}
 \sum_{k=1}^\lambda\frac{(-1)^{\lambda+k}(v+2k-1)!}{(2k-1)!(\lambda-k)!k!}
 &= \frac{v!}{\lambda!}\sum_{k=0}^\lambda\binom{\lambda}{k}\binom{v-1}{v-\lambda-k}2^{\lambda-k}\\
 &=2^\lambda \binom{v-1}{v-\lambda}\sum_{k=0}^{\lambda}
 \frac{\binom{\lambda}{k}\binom{v-\lambda}{k}}{\binom{\lambda+k-1}{k}}2^{-k},
\end{align*}
which is identical to the sequence~\cite[Proposition~7]{castiglione} and the sequence A181289 in the OEIS (written in slightly
different form). However, we were unable to find a closed form expression
for the sum; neither~\cite{gould} nor~\cite{gouldbook} shed any
light. The hope is to obtain an alternative bound for \eqref{eq:step3}
that does not involve the factor $e^{|\bsnu|}$. This is open for future
research.

As an alternative approach to the presented bootstrapping method, holomorphy arguments can be used to derive similar regularity bounds, see, e.g.,~\cite{cohendevore}.
\end{remark}

\begin{theorem} \label{thm:expbound-q1}
Let $\theta>0$, $\alpha_1,\alpha_2 \geq 0$, with $\alpha_1 + \alpha_2 >
0$. Let $f = (z,u_0) \in \calY'$ and $\widehat{u} \in \calX$. For
every $\bsy \in U$, let $u^\bsy \in \calX$ be the solution of
\eqref{eq:model} and $\Phi^\bsy$ be as in~\eqref{eq:Phi}, and then let
$q^\bsy = (q_1^\bsy,q_2^\bsy) \in \calY$ be the solution of
\eqref{eq:weakdual} with $f_{\rm dual}$ given by \eqref{eq:f-dual2}. Then
for all $\bsnu\in\mathscr{F}$ we have
$$
 \big\|\partial^\bsnu_\bsy \big(\!\exp(\theta\,\Phi^\bsy)\,q_1^\bsy\big)\big\|_{L^2(V;I)}
\leq
 \big\|\partial^\bsnu_\bsy \big(\!\exp(\theta\,\Phi^\bsy)\,q^\bsy\big)\big\|_{\calY}
 \le \frac{\mu}{2}\, (|\bsnu|+2)!\, (e\bsb)^\bsnu,
$$
where the sequence $\bsb = (b_j)_{j\ge 1}$ is defined by~\eqref{eq:bj},
$\sigma$ is defined by \eqref{eq:sigma2} and
\[
  \mu := e^{\max(\sigma,\,\sigma e^2+2\sigma-1)}\,
  \Big(\frac{\alpha_1 + \alpha_2\,\|E_T\|_{\calX\to L^2(D)}}{\beta_1}\Big)
  \Big(\frac{\|f\|_{\calY'}}{\beta_1} + \|\widehat{u}\|_\calX\Big).
\]
\end{theorem}

\begin{proof}
Using the Leibniz product rule and Theorem~\ref{thm:expbound} 
with Theorem~\ref{thm:adjregularity}, we obtain
\begin{align*}
 &\big\|\partial^\bsnu_\bsy \big(\!\exp(\theta\,\Phi^\bsy)\,q^\bsy\big)\big\|_{\calY}
 \le \sum_{\bsm\leq\bsnu}\binom{\bsnu}{\bsm} \big|\partial^\bsm_\bsy\exp(\theta\Phi^\bsy)\big|\,
  \big\|\partial^{\bsnu-\bsm}_\bsy q^\bsy\big\|_{\calY} \\
 &\le \sum_{\bsm\leq\bsnu}\binom{\bsnu}{\bsm}\,
 e^{\max(\sigma,\sigma\,e^2+2\sigma-1)}\,|\bsm|!\,(e\bsb)^\bsm \\
 &\qquad \times
   \Big(\frac{\alpha_1 + \alpha_2\,\|E_T\|^{2}_{\calX\to L^2(D)}}{\beta_1}\Big)
  \Big(\frac{\|f\|_{\calY'}}{\beta_1} + \|\widehat{u}\|_\calX\Big)\,
  \bsb^{\bsnu-\bsm}\,(|\bsnu| -|\bsm| + 1)! \\
 &\le \mu\, (e\bsb)^\bsnu \sum_{\bsm\leq\bsnu}\binom{\bsnu}{\bsm} |\bsm|!\,(|\bsnu| -|\bsm| + 1)!
 \;=\, \mu\, (e\bsb)^\bsnu \frac{(|\bsnu|+2)!}{2},
\end{align*}
with the last equality due to \cite[Formula~(9.5)]{kuonuyenssurvey}.
 
\end{proof}

\section{Error analysis} \label{sec:analysis}

Let $z^*$ denote the solution of \eqref{eq:reduced_OCuU} and let
$z_{s,n}^*$ be the minimizer of
\[
 J_{s,n}(z) := \mathcal{R}_{s,n}(\Phi^\bsy_{s}(z)) + \frac{\alpha_3}{2} \|z\|^2_{L^2(V';I)},
\]
where $\Phi_s^{\bsy}(z) = \Phi^{(y_1,y_2,\ldots,y_s,0,0,\ldots)}(z)$ is
the truncated version of $\Phi^\bsy(z)$ defined in \eqref{eq:Phi}, and
$\mathcal{R}_{s,n}$ is an approximation of the risk measure $\mathcal{R}$,
for which the integrals over the parameter domain $U=
[-\frac{1}{2},\frac{1}{2}]^{\mathbb{N}}$ are replaced by $s$-dimensional
integrals over $U_s= [-\frac{1}{2},\frac{1}{2}]^{s}$ and then approximated
by an $n$-point randomly-shifted QMC rule:
\[
 \calR_{s,n}(\Phi^\bsy_{s}(z)) =
 \begin{cases}
 \displaystyle\frac{1}{n}\sum_{i=1}^n \Phi^{\bsy^{(i)}}_{s}(z) & \mbox{for expected value}, \\
 \displaystyle\frac{1}{\theta} \ln \Big( \frac{1}{n}\sum_{i=1}^n \exp\big(\theta\,\Phi^{\bsy^{(i)}}_{s}(z)\big) \Big)
 & \mbox{for entropic risk measure},
 \end{cases}
\]%
for $\theta \in (0,\infty)$, for carefully chosen QMC points $\bsy^{(i)}$,
$i=1,\ldots,n$, involving a uniformly sampled random shift
$\bsDelta\in [0,1]^s$, see Section~\ref{sec:qmc}.

We have seen in the proof of Lemma~\ref{lem:exist_unique} that the risk
measures considered in this manuscript are convex and the objective
function $J$, see \eqref{eq:J-reduced}, is thus strongly convex. It is
important to note that the $n$-point QMC rule preserves the convexity of
the risk measure, so $J_{s,n}$ is a strongly convex function, because it
is a sum of a convex and a strongly convex function. Therefore we have the
optimality conditions $\langle J_{s,n}'(z_{s,n}^*), z -
z^*_{s,n}\rangle_{L^2(V;I),L^2(V';I)} \geq 0$ for all $z \in \mathcal{Z}$
and thus in particular $\langle J_{s,n}'(z_{s,n}^*), z^* -
z^*_{s,n}\rangle_{L^2(V;I),L^2(V';I)} \geq 0$. Similarly, we have
$\langle J'(z^*), z - z^* \rangle_{L^2(V;I),L^2(V';I)} \geq 0$, and in
particular $\langle -J'(z^*), z^* - z^*_{s,n} \rangle_{L^2(V;I),L^2(V';I)}
\geq 0$. Adding these inequalities gives
\begin{align*}
	\langle J'_{s,n}(z^*_{s,n}) - J'(z^*), z^* - z^*_{s,n}  \rangle_{L^2(V;I),L^2(V';I)} \geq 0\,.
\end{align*}
Hence
\begin{align*}
	&\alpha_3 \|z^* - z^*_{s,n} \|^2_{L^2(V';I)}\\
	&\leq \alpha_3 \|z^* - z^*_{s,n} \|^2_{L^2(V';I)} + \langle J'_{s,n}(z^*_{s,n}) - J'(z^*), z^* - z^*_{s,n}  \rangle_{L^2(V;I),L^2(V';I)} \\
	&= \langle J'_{s,n}(z^*_{s,n}) - \alpha_3 R_V^{-1}z^*_{s,n} - J'(z^*) + \alpha_3 R_V^{-1}z^*, z^* - z^*_{s,n}  \rangle_{L^2(V;I),L^2(V';I)}\\
	&= \langle J'_{s,n}(z^*_{s,n}) - \alpha_3 R_V^{-1}z^*_{s,n} - J'_{s,n}(z^*) + \alpha_3R_V^{-1}z^*  , z^* - z^*_{s,n}  \rangle_{L^2(V;I),L^2(V';I)} \\
	&\quad + \langle J'_{s,n}(z^*) - \alpha_3R_V^{-1}z^* - J'(z^*) + \alpha_3 R_V^{-1} z^*, z^* - z^*_{s,n} \rangle_{L^2(V;I),L^2(V';I)} \\
	&\leq \langle J'_{s,n}(z^*) - \alpha_3 R_V^{-1}z^* - J'(z^*) + \alpha_3 R_V^{-1}z^*, z^* - z^*_{s,n}  \rangle_{L^2(V;I),L^2(V';I)} \\
	&\leq \| J'_{s,n}(z^*) - \alpha_3R_V^{-1}z^* - J'(z^*) + \alpha_3 R_V^{-1}z^*\|_{L^2(V;I)} \|z^* - z^*_{s,n} \|_{L^2(V';I)} \,,
\end{align*}
where we used the $\alpha_3$-strong convexity of $J'_{s,n}$ in the fourth
step, i.e.,
\begin{align*}
\langle J'_{s,n}(z^*_{s,n}) - J'_{s,n}(z^*) - \alpha_3 R_V^{-1}(z^* - z_{s,n}^*), z^* - z_{s,n}^*\rangle_{L^2(V;I),L^2(V';I)} \leq 0\,.
\end{align*}
Thus we have with \eqref{eq:reduced_OCuU}
\begin{align*}
 \|z^* - z^*_{s,n}\|_{L^2(V';I)}
 &\leq \frac{1}{\alpha_3} \|J'(z^*) - J'_{s,n}(z^*)\|_{L^2(V;I)}.
\end{align*}

We will next expand this upper bound in order to split it into the
different error contributions: dimension truncation error and QMC error.
The different error contributions are then analyzed separately in the
following subsections for both risk measures.

In the case of the expected value, it follows from
\eqref{eq:gradient} that
\begin{align}\label{eq:split_linear}
 &\mathbb{E}_{\bsDelta} \|z^* - z^*_{s,n}\|^2_{L^2(V';I)}
 \le \frac{1}{\alpha_3^2} \mathbb{E}_{\bsDelta} \Big\| \int_U q_{1}^\bsy\,\mathrm d\bsy - \frac{1}{n} \sum_{i=1}^n q_{1,s}^{\bsy^{(i)}} \Big\|^2_{L^2(V;I)} \notag\\
 &\le \frac{{2}}{\alpha_3^2} \Big\| \int_U (q_{1}^\bsy - q_{1,s}^\bsy) \,\mathrm d\bsy \Big\|_{L^2(V;I)}^2
 + \frac{{2}}{\alpha_3^2} \mathbb{E}_{\bsDelta}\Big\| \int_{U_s} q_{1,s}^\bsy\,\mathrm d\bsy
   - \frac{1}{n} \sum_{i=1}^n q_{1,s}^{\bsy^{(i)}} \Big\|_{L^2(V;I)}^2,
\end{align}
where $q^\bsy_{1,s} := q^{(y_1,y_2,\ldots,y_s,0,0,\ldots)}_{1}$
denotes the truncated version of $q_1^\bsy$, and $\bbE_\bsDelta$
denotes the expected value with respect to the random shift $\bsDelta\in
[0,1]^s$.

In the case of the entropic risk measure, we recall that $J'(z)$ is
given by \eqref{eq:gradient_entropic}.
Let
\begin{align*}
T &:= \int_U \exp{\big(\theta\, \Phi^{\bsy}(z^*)\big)}\,\mathrm d\bsy\,, \qquad\quad\,
T_{s,n} := \frac{1}{n} \sum_{i=1}^n \exp{\big(\theta\, \Phi^{\bsy^{(i)}}_{s}(z^*_{s,n})\big)}\,,\\
S &:= \int_U \exp{\big(\theta\, \Phi^{\bsy}(z^*)\big)}\,q_1^\bsy(z^*)\,\mathrm d\bsy\,,\,\,
S_{s,n} := \frac{1}{n} \sum_{i=1}^n \exp{\big(\theta\, \Phi^{\bsy^{(i)}}_{s}(z^*_{s,n})\big)}\,q_{1,s}^{\bsy^{(i)}}(z^*_{s,n}),
\end{align*}
then we have
\begin{align*}
  \alpha_3\big\|z^* - z_{s,n}^*\big\|_{L^2(V';I)}
  &\le \Big\| \frac{S}{T} - \frac{S_{s,n}}{T_{s,n}} \Big\|_{L^2(V;I)}
  = \frac{\big\|S\,T_{s,n} - S_{s,n}\,T\big\|_{L^2(V;I)}}{T\,T_{s,n}} \\
  &= \frac{\big\|S\,T_{s,n} - S\,T + S\,T - S_{s,n}\,T\big\|_{L^2(V;I)}}{T\,T_{s,n}} \\
  &\le \frac{\big\|S\big\|_{L^2(V;I)}\,\big|T-T_{s,n}\big|}{T\,T_{s,n}}
           + \frac{\big\|S - S_{s,n}\big\|_{L^2(V;I)}}{T_{s,n}} \\
  &\le \mu\,\big|T-T_{s,n}\big| + \big\|S - S_{s,n}\big\|_{L^2(V;I)},
\end{align*}
where we used $T\ge 1$ and $T_{s,n}\ge 1$ and,  using the abbreviated notation
$g^\bsy(\bsx,t) := \exp(\theta\,\Phi^{\bsy}(z))\,q_1^{\bsy}(\bsx,t)$ we get
\begin{align*}
  \|S\|_{L^2(V;I)}^2
  &\,=\, \int_I \Big\|\int_U g^\bsy(\cdot,t) \,\rd\bsy \Big\|_V^2\,\rd t
  \,=\, \int_I \int_D \Big| \nabla \Big(\int_U g^\bsy(\bsx,t) \,\rd\bsy\Big) \Big|^2\,\rd\bsx\,\rd t \\
  &\,\le\, \int_U \int_I \int_D \big| \nabla g^\bsy(\bsx,t) \big|^2\,\rd\bsx\,\rd t\,\rd\bsy
  \,=\, \int_U \big\| g^\bsy \big\|^2_{L^2(V;I)}\,\rd\bsy 
  \,\le\, \mu^2\,,
\end{align*}
where we used Theorem~\ref{thm:expbound-q1} with $\bsnu=\bszero$.

We can write
\begin{align} \label{eq:split}
  \bbE_\bsDelta \Big\| \frac{S}{T} - \frac{S_{s,n}}{T_{s,n}} \Big\|_{L^2(V;I)}^2
  &\,\le\, 2\mu^2\,\bbE_\bsDelta \big|T-T_{s,n}\big|^2 + 2\bbE_\bsDelta \big\|S - S_{s,n}\big\|_{L^2(V;I)}^2.
\end{align}
For the first term on the right-hand side of \eqref{eq:split} we obtain
\begin{align}\label{eq:splitT}
  \bbE_\bsDelta\big|T-T_{s,n}\big|^2
  &\,\le\, {2}\big|T-T_s\big|^2 + {2}\bbE_\bsDelta \big|T_s-T_{s,n}\big|^2,
\end{align}
and for the second term we have
\begin{align}\label{eq:splitS}
  &\bbE_\bsDelta \|S-S_{s,n}\|_{L^2(V;I)}^2
  \,\le\, {2} \|S-S_s\|_{L^2(V;I)}^2 + {2} \bbE_\bsDelta \|S_s-S_{s,n}\|_{L^2(V;I)}^2.
\end{align}

\begin{remark}\label{rem:L2-Y}
Since we have $\|v_1\|_{L^2(V;I)} \leq \|v\|_{\calY}$ for all
$v=(v_1,v_2)\in \calY$ by definition, and thus in particular $\|\int_U
(q_1^\bsy - q_{1,s}^\bsy)\,\mathrm d\bsy \|_{L^2(V;I)} \leq \| \int_U
(q^\bsy - q_s^\bsy) \,\mathrm d\bsy \|_{\calY}$, we can replace
$q_1^\bsy,q_{1,s}^\bsy \in L^2(V;I)$ in \eqref{eq:split_linear} and
\eqref{eq:splitS} by $q^\bsy,q_s^\bsy \in \calY$. In order to obtain error
bounds and convergence rates for \eqref{eq:split_linear} and
\eqref{eq:splitS}, it is then sufficient to derive the results in the
$\calY$-norm, which is slightly stronger than the $L^2(V;I)$-norm.
\end{remark}

\subsection{Truncation error} \label{sec:trunc}

In this section we derive bounds and convergence rates for the errors that
occur by truncating the dimension, i.e., for the first terms in
\eqref{eq:split_linear}, \eqref{eq:splitT} and \eqref{eq:splitS}.

We prove a new and very general theorem for the truncation error based
on knowledge of regularity. The idea of the proof is based on a Taylor
series expansion and is similar to the approach in
\cite[Theorem~4.1]{GGKSS2019}. The use of Taylor series for dimension truncation error analysis has also been considered, for instance, in~\cite{matern1,GantnerPeters}.

\begin{theorem}\label{thm:dim_trunc}
Let $Z$ be a separable Banach space and let $g(\bsy): U \to Z$ be
analytically dependent on the sequence of parameters $\bsy \in U =
[-\frac{1}{2},\frac{1}{2}]^{\mathbb{N}}$. Suppose there exist
constants $C_0>0$, $r_1\geq0$, $r_2>0$, and a sequence $\bsrho =
(\rho_j)_{j\geq 1} \in \ell^p(\mathbb{N})$ for $0<p<1$, with $\rho_1 \geq
\rho_2 \geq \cdots$, such that for all $\bsy \in U$ and $\bsnu \in
\mathscr{F}$ we have
$$
\|\partial_{\bsy}^{\bsnu} g(\bsy)\|_Z \leq C_0\, (|\bsnu|+r_1)!\, (r_2\bsrho)^{\bsnu}.
$$%
Then, denoting $(\bsy_{\leq s};\boldsymbol{0}) =
(y_1,y_2,\ldots,y_s,0,0,\ldots)$, we have for all $s \in \mathbb{N}$
\begin{align*}
\Big\| \int_U \big(g(\bsy) - g(\bsy_{\leq s};\boldsymbol{0})\big)\,\mathrm d\bsy \Big\|_Z
 \le C_0\,C\,s^{-2/p+1}\,,
\end{align*}
for $C>0$ independent of $s$.
\end{theorem}

\begin{proof}
Let $\bsy\in U$ and $G\in Z'$ with $\|G\|_{Z'}\leq 1$ and define
$$
F(\bsy):=\langle G,g(\bsy) \rangle_{Z',Z}.
$$
Evidently, $\partial_{\bsy}^{\bsnu}F(\bsy)=\langle
G,\partial_{\bsy}^{\bsnu}g(\bsy)\rangle_{Z',Z}$ for all $\bsnu\in\mathscr
F$. Moreover,
$$
\sup_{\bsy\in U}|\partial_{\bsy}^{\bsnu}F(\bsy)|\leq C_0 (|\bsnu|+r_1)! (r_2\bsrho)^{\bsnu}\quad\text{for all}~\bsnu\in\mathscr F.
$$
For arbitrary $k\ge 1$ we consider the Taylor expansion of $F$ about the point
$(\bsy_{\leq s};\mathbf 0) = (y_1,y_2,\ldots,y_s,0,0,\ldots)$:
\begin{align*}
F(\bsy)&=F(\bsy_{\leq s};\mathbf 0)+\sum_{\ell=1}^k \sum_{\substack{|\bsnu|=\ell\\ \nu_j=0~\forall j\leq s}}\frac{\bsy^{\bsnu}}{\bsnu!}\partial_{\bsy}^{\bsnu}F(\bsy_{\leq s};\mathbf 0)\\
&\quad +\sum_{\substack{|\bsnu|=k+1\\ \nu_j=0~\forall j\leq s}}\frac{k+1}{\bsnu!}\,\bsy^{\bsnu}\int_0^1(1-t)^k\partial_{\bsy}^{\bsnu}F(\bsy_{\leq s};t\bsy_{>s})\,{\rm d}t.
\end{align*}
Rearranging this equation and integrating over $\bsy\in U$ yields
\begin{align} \label{eq:taylorstep}
 &\int_U (F(\bsy)-F(\bsy_{\leq s};\mathbf 0))\,{\rm d}\bsy
 =\sum_{\ell=1}^k \sum_{\substack{|\bsnu|=\ell\\ \nu_j=0~\forall j\leq s}}\frac{1}{\bsnu!}
 \int_U \bsy^{\bsnu}\,\partial_{\bsy}^{\bsnu}F(\bsy_{\leq s};\mathbf 0)\,{\rm d}\bsy
 \nonumber\\
 &\qquad + \sum_{\substack{|\bsnu|=k+1\\ \nu_j=0~\forall j\leq s}}\frac{k+1}{\bsnu!}
 \int_U\int_0^1(1-t)^k\,\bsy^{\bsnu}\,\partial_{\bsy}^{\bsnu}F(\bsy_{\leq s};t\,\bsy_{>s})\,{\rm d}t\,{\rm d}\bsy.
\end{align}
If there is any component $\nu_j=1$ with $j>s$, then the summand in
the first term vanishes, since (for all $\bsnu \in \mathscr F$ with $\nu_j
= 0$ $\forall j\leq s$)
$$
 \int_U \bsy^{\bsnu}\partial_{\bsy}^{\bsnu}F(\bsy_{\leq s};\mathbf 0)\,{\rm d}\bsy
 =\int_{U_{\leq s}} \partial_{\bsy}^{\bsnu}F(\bsy_{\leq s};\mathbf 0)\,
 \underset{=0 \text{ if at least one } \nu_j = 1}{\underbrace{\Big(\prod_{j>s}\int_{-1/2}^{1/2}y_j^{\nu_j}\,{\rm d}y_j\Big)}}\,{\rm d}\bsy_{\leq s}=0,
$$
where we used Fubini's theorem. Taking the absolute value on both sides in~\eqref{eq:taylorstep} and using $|y_j|\leq \frac12$, we obtain
\begin{align}
&\Big|\int_U (F(\bsy)-F(\bsy_{\leq s};\mathbf 0))\,{\rm d}\bsy\Big|\notag\\
&\leq \sum_{\ell=2}^k \!\!\!\sum_{\substack{|\bsnu|=\ell\\ \nu_j=0~\forall j\leq s\\ \nu_j\neq 1~\forall j>s}}\!\!\frac{1}{2^{\bsnu}\bsnu!}\sup_{\bsy\in U}|\partial_{\bsy}^{\bsnu}F(\bsy)|
 +\!\!\!\!\sum_{\substack{|\bsnu|=k+1\\ \nu_j=0~\forall j\leq s}}\!\!\frac{k+1}{2^{\bsnu}\bsnu!} \int_0^1 (1-t)^k \sup_{\bsy\in U}|\partial_{\bsy}^{\bsnu}F(\bsy)|\, \mathrm dt\notag\\
&= \sum_{\ell=2}^k \!\!\!\sum_{\substack{|\bsnu|=\ell\\ \nu_j=0~\forall j\leq s\\ \nu_j\neq 1~\forall j>s}}\frac{1}{2^{\bsnu}\bsnu!}\sup_{\bsy\in U}|\partial_{\bsy}^{\bsnu}F(\bsy)|
 + \!\!\!\! \sum_{\substack{|\bsnu|=k+1\\ \nu_j=0~\forall j\leq s}}\frac{1}{2^{\bsnu}\bsnu!} \sup_{\bsy\in U}|\partial_{\bsy}^{\bsnu}F(\bsy)|\notag\\
&\leq \sum_{\ell=2}^k \!\!\! \sum_{\substack{|\bsnu|=\ell\\ \nu_j=0~\forall j\leq s\\ \nu_j\neq 1~\forall j>s}}\frac{1}{2^{\bsnu}\bsnu!} C_0\,(|\bsnu|+r_1)!\, (r_2\bsrho)^{\bsnu}
 + \!\!\!\!\sum_{\substack{|\bsnu|=k+1\\ \nu_j=0~\forall j\leq s}}\frac{1}{2^{\bsnu}\bsnu!} C_0\, (|\bsnu|+r_1)!\, (r_2\bsrho)^{\bsnu}\notag\\
&\leq C_0\,(k+r_1)!\,\frac{{r_2}^k}{2^2\,2!} \sum_{\ell=2}^k \!\!\!
 \sum_{\substack{|\bsnu|=\ell\\ \nu_j=0~\forall j\leq s\\ \nu_j\neq 1~\forall j>s}}\!\!\!\! \bsrho^{\bsnu}
 +C_0\,(k+1+r_1)!\,\left(\frac{{r_2}}{2}\right)^{k+1} \!\!\!\!\!\!\!
 \sum_{\substack{|\bsnu|=k+1\\ \nu_j=0~\forall j\leq s}} \!\!\!\!\!\frac{\bsrho^{\bsnu}}{\bsnu!}.\label{eq:presup}
\end{align}
Furthermore, we have 
\begin{align*}
\Big\|\int_U (g(\bsy) - g(\bsy_{\leq s};\boldsymbol{0}))\,{\rm d}\bsy\Big\|_{Z}
&= \sup_{\substack{G\in Z'\\ \|G\|_{Z'}\leq 1}}\Big|\Big\langle G,\int_U (g(\bsy) - g(\bsy_{\leq s};\boldsymbol{0}))\,{\rm d}\bsy\Big\rangle_{Z',Z}\Big|\\
&=\sup_{\substack{G\in Z'\\ \|G\|_{Z'}\leq 1}}\Big|\int_U\langle G,g(\bsy) - g(\bsy_{\leq s};\boldsymbol{0})\rangle_{Z',Z}\,{\rm d}\bsy\Big| \\
&=\sup_{\substack{G\in Z'\\ \|G\|_{Z'}\leq 1}}\Big|\int_U (F(\bsy)-F(\bsy_{\leq s};\mathbf 0))\,{\rm d}\bsy\Big|\,,
\end{align*}
which is also bounded by the last expression in \eqref{eq:presup}.

For $s$ sufficiently large, we obtain in complete analogy to
\cite{Gantner} that the first term in \eqref{eq:presup} satisfies
\begin{align*}
\sum_{\ell=2}^k \sum_{\substack{|\bsnu|=\ell\\ \nu_j=0~\forall j\leq s\\ \nu_j\neq 1~\forall j>s}}\bsrho^{\bsnu}
&=\sum_{\substack{2\le |\bsnu|\leq k\\ \nu_j=0~\forall j\leq s\\ \nu_j\neq 1~\forall j>s}}\bsrho^{\bsnu}
 \leq \sum_{\substack{0 \neq |\bsnu|_\infty \leq k\\ \nu_j=0~\forall j\leq s\\ \nu_j\neq 1~\forall j>s}}\bsrho^{\bsnu}
 =-1+\prod_{j>s}\Big(1+\sum_{\ell=2}^k \rho_j^\ell\Big)\\
&=-1+\prod_{j>s}\Big(1+\frac{1-\rho_j^{k-1}}{1-\rho_j}\,\rho_j^2\Big)=\mathcal O(s^{-2/p+1}),
\end{align*}
since $\bsrho\in \ell^p$ with $0<p<1$ and $\rho_1\geq
\rho_2\geq \cdots$ by assumption.

On the other hand, we can use the multinomial theorem to bound
the second term in \eqref{eq:presup}
$$
\sum_{\substack{|\bsnu|=k+1\\ \nu_j=0~\forall j\leq s}} \frac{\bsrho^{\bsnu} }{\bsnu!}
\leq \sum_{\substack{|\bsnu|=k+1\\ \nu_j=0~\forall j\leq s}} \binom{k+1}{\bsnu}\bsrho^{\bsnu}
=\bigg(\sum_{j>s}\rho_j\bigg)^{k+1}
=\mathcal O(s^{(k+1)(-1/p+1)}),
$$
where we used $\sum_{j>s}\rho_j\leq s^{-1/p+1} (\sum_{j=1}^\infty \rho_j^p
)^{1/p}$. (The last inequality follows directly from
\cite[Lemma~5.5]{CohenDeVoreSchwab}, which is often attributed to
Stechkin. For an elementary proof we refer to \cite[Lemma
3.3]{KressnerTobler}.)

Taking now $k=\lceil \frac{1}{1-p}\rceil$ yields that
\eqref{eq:presup} is of order $\calO(s^{-2/p+1})$. Note that $k\geq 2$
for $0<p<1$. The result can be extended to all $s$ by a trivial adjustment
of the constants.  
\end{proof}

\begin{remark}
The assumption of analyticity of the integrand can be replaced since the
Taylor series representation remains valid under the weaker assumption
that only the $\bsnu$-th partial derivatives with $|\bsnu|\leq k+1$ for
$k=\lceil \frac{1}{1-p}\rceil$ and $0<p<1$ exist.
\end{remark}

We now apply this general result to the first terms in
\eqref{eq:split_linear}, \eqref{eq:splitT} and \eqref{eq:splitS}.

\begin{theorem} \label{thm:trunc}
Let $\theta>0$, $\alpha_1,\alpha_2 \geq 0$, with $\alpha_1 + \alpha_2 >
0$. Let $f = (z,u_0) \in \calY'$ and $\widehat{u} \in \calX$. For
every $\bsy \in U$, let $u^\bsy \in \calX$ be the solution of
\eqref{eq:model} and $\Phi^\bsy$ be as in~\eqref{eq:Phi}, and then let
$q^\bsy \in \calY$ be the solution of \eqref{eq:weakdual} with $f_{\rm
dual}$ given by \eqref{eq:f-dual2}. Suppose the sequence $\bsb =
(b_j)_{j\ge 1}$ defined by~\eqref{eq:bj} satisfies
\begin{align}
 &\sum_{j\ge 1} b_j^p < \infty \quad\mbox{for some } 0 < p < 1, \quad\mbox{and} \label{eq:dim_assumption1}
 \\
 &b_1\ge b_2 \ge \cdots. \label{eq:dim_assumption2}
\end{align}
Then for every $s \in \mathbb{N}$, the truncated solutions $u_s^\bsy$, $q_s^\bsy$ and
$\Phi_s^\bsy$ satisfy
\begin{align*}
 \Big\| \int_U (u^\bsy - u_s^\bsy) \,\rd\bsy \Big\|_{{\calX}}
 &\leq C\, s^{-2/p+1},
 \\
  \Big\| \int_U (q^\bsy - q_s^\bsy) \,\rd\bsy \Big\|_{\calY}
 &\leq C\, s^{-2/p+1},
 \\
 \|S - S_s\|_{\calY} =
 \Big\| \int_U \big(\exp{\big(\theta\, \Phi^{\bsy}\big)}\,q^\bsy - \exp{\big(\theta\,  \Phi^\bsy_s\big)}\,q_s^\bsy\big) \,\rd\bsy \Big\|_{\calY}
 &\leq C\, s^{-2/p+1},
 \\
 |T - T_s| =
 \Big| \int_U \big(\exp{\big(\theta\, \Phi^{\bsy}\big)} - \exp{\big(\theta\,  \Phi^\bsy_s\big)}\big) \,\rd\bsy \Big|
 &\leq C\, s^{-2/p+1}.
\end{align*}
In each case we have a generic constant $C>0$ independent of $s$, but
depending on $z$, $u_0$, $\widehat{u}$ and other constants as
appropriate.
\end{theorem}

\begin{proof}
The result is a corollary of Theorem~\ref{thm:dim_trunc} by applying the
regularity bounds in Lemma~\ref{lem:statebound}, Theorem~\ref{thm:adjregularity},
Theorem~\ref{thm:expbound-q1} and Theorem~\ref{thm:expbound}.  
\end{proof}

\subsection{Quasi-Monte Carlo error} \label{sec:qmc}

We are interested in computing $s$-dimensional Bochner integrals of the form
\begin{align*}
I_s(g):=\int_{U_s}g(\bsy)\,{\rm d}\bsy,
\end{align*}
where $g(\bsy)$ is an element of a separable Banach space $Z$ for each $\bsy \in U_s$.
As our estimator of $I_s(g)$, we use a cubature rule of the form
$$
Q_{s,n}(g) := \sum_{i=1}^n \alpha_i \,g(\bsy^{(i)}).
$$
with weights $\alpha_i \in \mathbb{R}$ and cubature points $\bsy^{(i)} \in
U_s$. In particular, we are interested in QMC rules (see, e.g.,
\cite{dick2013high,kuonuyenssurvey}), which are cubature rules
characterized by equal weights $\alpha_i = 1/n$ and carefully chosen
points $\bsy^{(i)}$ for $i=1,\ldots,n$.

We shall see that for sufficiently smooth integrands, randomly shifted
lattice rules lead to convergence rates not depending on the dimension,
and which are faster compared to Monte Carlo methods. Randomly shifted
lattice rules are QMC rules with cubature points given by
\begin{align}
\bsy^{(i)}_{\boldsymbol{\Delta}} :=
\mathrm{frac}\Big(\frac{i\boldsymbol{z}}{n}+\boldsymbol{\Delta}\Big) - \Big(\frac{1}{2},\ldots,\frac{1}{2}\Big),\label{eq:qmc1}
\end{align}
where $\boldsymbol{z}\in \mathbb{N}^s$ is known as the generating vector,
$\boldsymbol{\Delta} \in [0,1]^s$ is the random shift and $\mathrm{frac}(\cdot)$
means to take the fractional part of each component in the vector. In
order to get an unbiased estimator, in practice we take the mean over
$R$ uniformly drawn random shifts, i.e., we estimate $I_s(g)$ using
\begin{align}
\overline{Q}(g):=\frac{1}{R}\sum_{r=1}^R Q^{(r)}(g),\quad\text{with}\quad Q^{(r)}(g):=\frac{1}{n}\sum_{i=1}^ng(\bsy^{(i)}_{\boldsymbol{\Delta}^{(r)}}).\label{eq:qmc2}
\end{align}

In this section we derive bounds and convergence rates for the errors that
occur by applying a QMC method for the approximation of the integrals in
the second terms in \eqref{eq:split_linear}, \eqref{eq:splitT} and
\eqref{eq:splitS}. We first prove a new general result which holds for
any cubature rule in a separable Banach space setting.

\begin{theorem}\label{thm:gencub} Let $U_s = [-\frac{1}{2},\frac{1}{2}]^s$ and let
$\calW_s$ be a Banach space of functions $F: U_s\to \bbR$, which is
continuously embedded in the space of continuous functions. Consider an
$n$-point cubature rule with weights $\alpha_i\in\bbR$ and points
$\bsy^{(i)}\in U_s$, given by
\[
 I_s(F) \,:=\, \int_{U_s} F(\bsy)\,\rd\bsy
 \,\approx\, \sum_{i=1}^n \alpha_i\, F(\bsy^{(i)})
 \,=:\, Q_{s,n}(F),
\]
and define the worst case error of $Q_{s,n}$ in $\calW_s$ by
\[
  e^{\rm wor}(Q_{s,n};\calW_s) \,:=\,
  \sup_{F\in\calW_s,\, \|F\|_{\calW_s}\le 1} |I_s(F) - Q_{s,n}(F)|.
\]
Let $Z$ be a separable Banach space and let $Z'$ denote its dual space.
Let $g: \bsy \mapsto g(\bsy)$ be continuous and $g(\bsy) \in Z$ for all
$\bsy \in U_s$. Then
\begin{align} \label{eq:general}
 \Big\|\int_{U_s} g(\bsy) \,\rd\bsy
  - \sum_{i=1}^n \alpha_i\, g(\bsy^{(i)}) \Big\|_Z
 \,\le\, e^{\rm wor}(Q_{s,n};\calW_s)\,
 \sup_{\satop{G\in Z'}{\|G\|_{Z'}\le 1}} \| G(g)\|_{\calW_s}.
\end{align}
\end{theorem}

\begin{proof}
From the separability of $Z$ and the continuity of $g(\bsy)$ we get strong
measurability of $g(\bsy)$. Moreover, from the compactness of $U_s$ and
the continuity of $\bsy \mapsto g(\bsy)$ we conclude that $\sup_{\bsy \in
U_s}\|g(\bsy)\|_Z < \infty$ and hence $\int_{U_s} \|g(\bsy)\|_Z\, \rd\bsy
< \infty$, which in turn implies $\| \int_{U_s} g(\bsy)\, \rd\bsy\|_Z
<\infty$. Thus $g(\bsy)$ is Bochner integrable.

Furthermore, for every normed space $Z$, its dual space $Z'$ is a Banach
space equipped with the norm $\|G\|_{Z'} := \sup_{g \in Z,\, \|g\|_{Z}
\leq 1} |\langle G,g\rangle_{Z',Z}|$. Then it holds for every $g\in Z$
that $\|g\|_Z = \sup_{G \in Z',\, \|G\|_{Z'} \leq 1} |\langle G,
g\rangle_{Z',Z}|$. This follows from the Hahn--Banach Theorem, see, e.g.,
\cite[Theorem 4.3]{rudin}.

Thus we have
\begin{align} \label{eq:supbound}
 &\Big\| \int_{U_s} g(\bsy)\, \rd\bsy - \sum_{i=1}^n \alpha_i\, g(\bsy^{(i)}) \Big\|_Z \nonumber\\
 &\qquad= \sup_{\stackrel{G\in Z'}{\|G\|_{Z'}\leq 1}}
 \Big| \Big\langle G, \int_{U_s} g(\bsy)\, \rd\bsy
 - \sum_{i=1}^n \alpha_i\, g(\bsy^{(i)}) \Big\rangle_{Z',Z} \Big| \notag\\
 &\qquad= \sup_{\stackrel{G\in Z'}{\|G\|_{Z'}\leq 1}}
 \Big| \int_{U_s} \langle G,g(\bsy)\rangle_{Z',Z}\, \rd\bsy
 - \sum_{i=1}^n \alpha_i\, \langle G,g(\bsy^{(i)}) \rangle_{Z',Z} \Big| \,,
\end{align}
where we used the linearity of $G$ and the fact that for Bochner integrals
we can swap the integral with the linear functional, see, e.g.,
\cite[Corollary V.2]{Yosida}.

From the definition of the worst case error of $Q_{s,n}$ in $\calW_s$ it
follows that for any $F\in \calW_s$ we have
\[
  |I_s(F) - Q_{s,n}(F)| \,\le\, e^{\rm wor}(Q_{s,n};\calW_s)\,\|F\|_{\calW_s}.
\]
Applying this to the special case $F(\bsy) = G(g(\bsy)) = \langle G,
g(\bsy)\rangle_{Z',Z}$ in \eqref{eq:supbound} yields \eqref{eq:general}.
 
\end{proof}

\begin{theorem}\label{thm:MSE}
Let the assumptions of the preceding Theorem hold. In addition,
suppose there exist constants $C_0>0$, $r_1\ge 0$, $r_2>0$ and a positive
sequence $\bsrho = (\rho_j)_{j\ge 1}$ such that for all $\setu \subseteq
\{1:s\}$ and for all $\bsy \in U_s$ we have
\begin{align} \label{eq:ass-u}
 \Big\| \frac{\partial^{|\setu|}}{\partial \bsy_{\setu}} g(\bsy)\Big\|_Z
 \leq C_0\, (|\setu|+r_1)!\, \prod_{j \in \setu} (r_2\,\rho_j).
\end{align}
Then, a randomly shifted lattice rule can be constructed using a component-by-component (CBC)
algorithm such that
\begin{align*}
    \bbE_\bsDelta \Big\|\int_{U_s} g(\bsy) \,\rd\bsy
  - \frac{1}{n} \sum_{i=1}^n g({\bsy^{(i)}}) \Big\|_{Z}^2
  \,\le\, C_{s,\bsgamma,\lambda}\, [\phi_{\rm tot}(n)]^{-1/\lambda}
  \,\,\,\mbox{for all}\,\,\, \lambda\in (\tfrac{1}{2},1],
\end{align*}
where $\phi_{\rm tot}(n)$ is the Euler totient function, with $1/\phi_{\rm
tot}(n)\le 2/n$ when $n$ is a prime power, and
\begin{align} \label{eq:lattice-bound}
 C_{s,\bsgamma,\lambda}
 := C_0^2
 \bigg( \sum_{\emptyset\ne\setu\subseteq\{1:s\}} \!\!\!\gamma_\setu^\lambda
 \bigg(\frac{2\zeta(2\lambda)}{(2\pi^2)^\lambda} \bigg)^{|\setu|}\bigg)^{\frac{1}{\lambda}}
 \bigg( \sum_{\setu\subseteq\{1:s\}} \!\!\!\frac{[(|\setu|+r_1)!]^2 
 \prod_{j\in\setu} (r_2\,\rho_j)^2}{\gamma_\setu} \bigg).
\end{align}
\end{theorem}

\begin{proof}
We consider randomly shifted lattice rules in the unanchored weighted
Sobolev space $\calW_{s,\bsgamma}$ with norm
\begin{align*}
 \|F\|_{\calW_{s,\bsgamma}}^2
 :=& \!\sum_{\setu\subseteq\{1:s\}} \!\frac{1}{\gamma_\setu}
 \int_{[-\frac{1}{2},\frac{1}{2}]^{|\setu|}} \Big|
 \int_{[-\frac{1}{2},\frac{1}{2}]^{s-|\setu|}}\! \frac{\partial^{|\setu|}}{\partial \bsy_\setu}
 F(\bsy_\setu;\bsy_{\{1:s\}\setminus\setu})\, \rd\bsy_{\{1:s\}\setminus\setu} \Big|^2 \rd\bsy_\setu \\
 \le& \!\sum_{\setu\subseteq\{1:s\}} \frac{1}{\gamma_\setu}
 \int_{U_s} \Big| \frac{\partial^{|\setu|}}{\partial \bsy_\setu}
 F(\bsy)\, \Big|^2 \rd\bsy.
\end{align*}
It is known that CBC construction yields a lattice generating vector
satisfying
\[
  \bbE_\bsDelta [e^{\rm wor}(Q_{s,n};\calW)]^2
  \,\le\, \Big( \frac{1}{\phi_{\rm tot}(n)} \sum_{\emptyset\ne\setu\subseteq\{1:s\}}\!\! \gamma_{\setu}^\lambda
  \Big(\frac{2\zeta(2\lambda)}{(2\pi^2)^\lambda} \Big)^{|\setu|}\Big)^{\frac{1}{\lambda}}
  \,\,\mbox{for all}\,\, \lambda\in (\tfrac{1}{2},1].
\]

We have from \eqref{eq:general} that
\begin{align*}
    \bbE_\bsDelta \Big\|\int_{U_s} g(\bsy) \,\rd\bsy
  - \frac{1}{n} \sum_{i=1}^n g(\bsy^{(i)}) \Big\|_{Z}^2
  \,\le\, \bbE_\bsDelta [e^{\rm wor}(Q_{s,n};\calW)]^2
  \!\sup_{\stackrel{G\in Z'}{\|G\|_{Z'}\leq 1}}\! \|G(g)\|_{\calW_{s,\bsgamma}}^2.
\end{align*}
Using the definition of the $\calW_{s,\bsgamma}$-norm, we have
\begin{align*}
 &\|G(g)\|_{\calW_{s,\bsgamma}}^2
 \le \sum_{\setu\subseteq\{1:s\}} \frac{1}{\gamma_{\setu}}
 \int_{U_s} \Big| \frac{\partial^{|\setu|}}{\partial \bsy_{\setu}}  G(g(\bsy))\Big|^2 \, \rd\bsy\\
 &= \sum_{\setu\subseteq\{1:s\}} \frac{1}{\gamma_{\setu}}
 \int_{U_s} \Big| G\Big( \frac{\partial^{|\setu|}}{\partial \bsy_{\setu}}  g(\bsy)\Big)\Big|^2 \, \rd\bsy
 \le \!\sum_{\setu\subseteq\{1:s\}} \!\frac{1}{\gamma_{\setu}}
 \int_{U_s} \|G\|_{Z'}^2\, \Big\| \frac{\partial^{|\setu|}}{\partial \bsy_{\setu}} g(\bsy)\Big\|_Z^2\, \rd\bsy\,.
\end{align*}
We can now use the assumption \eqref{eq:ass-u} and combine all of the
estimates to arrive at the required bound.  
\end{proof}
\begin{remark}
    Theorem~\ref{thm:gencub} holds for arbitrary cubature rules, thus similar results to Theorem~\ref{thm:MSE} can be stated for other cubature rules. In particular, the regularity bounds obtained in Sections~\ref{sec:reg-adj} and~\ref{sec:reg-ent} can also be used for worst case error analysis of higher-order QMC quadrature as well as sparse grid integration.
\end{remark}
We now apply Theorem~\ref{thm:MSE} to the second terms in
\eqref{eq:split_linear}, \eqref{eq:splitT} and~\eqref{eq:splitS}.

\begin{theorem} \label{thm:qmc}
Let $\theta>0$, $\alpha_1,\alpha_2 \geq 0$, with $\alpha_1 + \alpha_2 >
0$. Let $f = (z,u_0) \in \calY'$ and $\widehat{u} \in \calX$. For
every $\bsy \in U$ and $s\in\bbN$, let $u^\bsy_s \in \calX$ be the
truncated solution of \eqref{eq:model} and $\Phi^\bsy_s$ be as
in~\eqref{eq:Phi}, and then let $q^\bsy_s \in \calY$ be the truncated
solution of \eqref{eq:weakdual} with $f_{\rm dual}$ given by
\eqref{eq:f-dual2}. Then a randomly shifted lattice rule can be
constructed using a CBC algorithm such that for all $\lambda\in
(\tfrac{1}{2},1]$ we have
\begin{align}
  \bbE_\bsDelta \Big\|\int_{U_s} u_{s}^\bsy\,\rd\bsy
  &- \frac{1}{n} \sum_{i=1}^n u_{s}^{\bsy^{(i)}} \Big\|_{{\calX}}^2
  \le C_{s,\bsgamma,\lambda}\, [\phi_{\rm tot}(n)]^{-1/\lambda},\label{eq:qmcerr1}
  \\
  \bbE_\bsDelta \Big\|\int_{U_s} q_{s}^\bsy\,\rd\bsy
  &- \frac{1}{n} \sum_{i=1}^n q_{s}^{\bsy^{(i)}} \Big\|_{\calY}^2
  \le C_{s,\bsgamma,\lambda}\, [\phi_{\rm tot}(n)]^{-1/\lambda},\label{eq:qmcerr2}
  \\
 \bbE_\bsDelta\|S_s - S_{s,n}\|_{L^2(V;I)}^2
 &\le \bbE_\bsDelta \Big\| \int_{U_s} \exp(\theta\,\Phi_{s}^{\bsy})\,q_s^\bsy \,\rd\bsy
  - \frac{1}{n} \sum_{i=1}^n \exp(\theta\,\Phi_{s}^{\bsy^{(i)}})\,q_s^{\bsy^{(i)}} \Big\|^2_\calY\notag \\
 &\le C_{s,\bsgamma,\lambda}\, [\phi_{\rm tot}(n)]^{-1/\lambda},\label{eq:qmcerr3}
 \\
 \bbE_\bsDelta|T_s - T_{s,n}|^2
 &\le \bbE_\bsDelta \Big| \int_{U_s} \exp(\theta\,\Phi_{s}^{\bsy}) \,\rd\bsy
  - \frac{1}{n} \sum_{i=1}^n \exp(\theta\,\Phi_{s}^{\bsy^{(i)}}) \Big|^2\notag \\
 &\le C_{s,\bsgamma,\lambda}\, [\phi_{\rm tot}(n)]^{-1/\lambda},\label{eq:qmcerr4}
\end{align}
where $\phi_{\rm tot}(n)$ is the Euler totient function, with $1/\phi_{\rm
tot}(n)\le 2/n$ when $n$ is a prime power. Here $C_{s,\bsgamma,\lambda}$
is given by \eqref{eq:lattice-bound}, with $r_1 = 2$, $r_2 = e$, $\rho_j =
b_j$ defined in \eqref{eq:bj}, and $C_0>0$ is independent of $s$, $n$,
$\lambda$ and weights $\bsgamma$ but depends on $u$, $z_0$,
$\widehat{u}$ and other constants.

With the choices
\begin{align*}
 \lambda
 &= \begin{cases}
 \frac{1}{2 - 2\delta} \text{ for all $\delta \in (0,1)$} & \text{if } p\in (0,\frac{2}{3}]\,, \\
 \frac{p}{2-p} & \text{if } p\in (\frac{2}{3},1)\,,
 \end{cases}
 \\
 \gamma_\setu &= \gamma_\setu^* :=
 \bigg( (|\setu|+r_1)!\, \prod_{j \in \setu}
 \frac{r_2\, \rho_j}{\sqrt{2\zeta(2\lambda)/(2\pi^2)^\lambda}} \bigg)^{2/(1+\lambda)}\,,
\end{align*}
we have that $C_{s,\bsgamma^*,\lambda}$ is bounded independently of $s$.
$($However, $C_{s,\bsgamma^*,\frac{1}{2-2\delta}} \to \infty$ as $\delta
\to 0$ and $C_{s,\bsgamma^*,\frac{p}{2-p}} \to \infty$ as $p \to
(2/3)^+$.$)$ Consequently, under the assumption
\eqref{eq:dim_assumption1}, the above three mean-square errors are of
order
\begin{align}
\kappa(n) :=
\begin{cases}\label{eq:kappa}
[\phi_{\rm tot}(n)]^{-(2-2\delta)} \,\,\,\,\text{ for all $\delta \in (0,1)$}& \text{if } p\in (0,\frac{2}{3}]\,,\\
[\phi_{\rm tot}(n)]^{-(2/p -1)} & \text{if } p\in (\frac{2}{3},1)\,.
\end{cases}
\end{align}
\end{theorem}

\begin{proof}
The mean-square error bounds are a corollary of Theorem~\ref{thm:MSE} by
applying the regularity bounds in Lemma~\ref{lem:statebound}, Theorem~\ref{thm:adjregularity},
Theorem~\ref{thm:expbound-q1} and Theorem~\ref{thm:expbound}. For
simplicity we set $C_0$, $r_1$ and $r_2$ to be the largest values arising
from the four results.

We know from \cite[Lemma~6.2]{KSS2012} that for any $\lambda$,
$C_{s,\bsgamma,\lambda}$ defined in \eqref{eq:lattice-bound} is minimized
by $\gamma_\setu = \gamma_\setu^*$. By inserting $\bsgamma^*$ into
$C_{s,\bsgamma,\lambda}$ we can then derive the condition $p <
\frac{2\lambda}{1+\lambda} <1$ for which $C_{s,\bsgamma^*,\lambda}$ is
bounded independently of $s$. This condition on $\lambda$, together with
$\lambda \in (\frac{1}{2},1]$ and $p \in (0,1)$ yields the result.  
\end{proof}

\subsection{Combined error bound}

Combining the results in this section gives the final theorem.

\begin{theorem}
Let $\alpha_1,\alpha_2 \geq 0$ and $\alpha_3>0$ with $\alpha_1+\alpha_2 >0$, and let the risk measure $\calR$ be either the expected value or the entropic risk measure with $\theta >0$. Denote by $z^*$ the unique solution of \eqref{eq:reduced_OCuU} and by $z^*_{s,n}$ the unique solution of the truncated problem using a randomly shifted lattice rule as approximation of the risk measure.  Then, if \eqref{eq:dim_assumption1} and \eqref{eq:dim_assumption2} hold, we have
\begin{align*}
 \sqrt{\bbE_\bsDelta \|z^* - z^*_{s,n}\|_{L^2(V';I)}^2}
 \leq C 
 \Big(s^{-2/p+1} + \sqrt{\kappa(n)}\Big)\,,
\end{align*}
where $\kappa(n)$ is given in \eqref{eq:kappa}, and the constant $C>0$
is independent of $s$ and $n$ but depends on $z$, $u_0$, $\widehat{u}$ and
other constants.
\end{theorem}

\begin{proof}
This follows from \eqref{eq:split_linear}--\eqref{eq:splitS},
Remark~\ref{rem:L2-Y}, Theorem~\ref{thm:trunc} and
Theorem~\ref{thm:qmc}.  
\end{proof}

\section{Numerical experiments} \label{sec:num}

We consider the coupled PDE system
\begin{align}
&\begin{cases}
\partial_t u^\bsy(\bsx,t)-\nabla \cdot (a^\bsy (\bsx)\nabla u^\bsy (\bsx,t))=z(\bsx,t)\\
u^\bsy(\bsx,t)=0\\
u^\bsy(\bsx,0)=u_0(\bsx)
\end{cases}\label{eq:coupled1}
\intertext{and}
&\begin{cases}
-\partial_t q^\bsy(\bsx,t)-\nabla \cdot (a^\bsy (\bsx)\nabla q^\bsy (\bsx,t))=\alpha_1R_V(u^\bsy (\bsx,t)-\widehat{u}(\bsx,t))\\
q^\bsy(\bsx,t)=0\\
q^\bsy(\bsx,T)=\alpha_2(u^\bsy(\bsx,T)-\widehat{u}(\bsx,T)),
\end{cases}\label{eq:coupled2}
\end{align}
where the first equations in \eqref{eq:coupled1} and \eqref{eq:coupled2} hold for $\bsx \in D$, $t\in I$, $\bsy\in U$, the second equations hold for $\bsx \in \partial D$, $t\in I$, $\bsy\in U$, and the last equations hold for $\bsx\in D$ and $\bsy\in U$. We fix the physical domain $D=(0,1)^2$ and the terminal time $T=1$. The uncertain diffusion coefficient, defined as in~\eqref{eq:axy}, is independent of $t$, and parameterized in all experiments with mean field $a_0(\bsx)\equiv 1$ and the fluctuations
$$
\psi_j(\bsx)=\frac12 j^{-\vartheta}\sin(\pi jx_1)\sin(\pi j x_2)\quad\text{for }\vartheta>1~\text{and}~j\in \mathbb N.
$$
We use the implicit Euler finite difference scheme with step size $\Delta t=\frac{T}{500}=2\cdot10^{-3}$ to discretize the PDE system with respect to the temporal variable. The spatial part of the PDE system is discretized using a first order finite element method with mesh size $h=2^{-5}$ and the Riesz operator in the loading term corresponding to~\eqref{eq:coupled2} can be evaluated using~\eqref{eq:RV}. In all experiments, the lattice rules are generated using the fast CBC algorithm with weights chosen as in Theorem~\ref{thm:qmc}, where we used the parameter value $\beta_1=1$ in~\eqref{eq:bj}.

In the numerical experiments presented in Subsections~\ref{sec:num_dimtrunc}--\ref{sec:num_oc}, we choose
\begin{align*}
\widehat{u}(\bsx,t) &:=
 \chi_{\|\bsx-(c_1(t),c_2(t))\|_{\infty}\leq \frac{1}{10}}(\bsx)\,\widehat{u}_1(\bsx,t)\\
 &\qquad+ \chi_{\|\bsx+(c_1(t),c_2(t))-(1,1)\|_{\infty}\leq \frac{1}{10}}(\bsx)\,\widehat{u}_2(\bsx,t),
\end{align*}
where
\begin{align*}
&\widehat{u}_1(\bsx,t):=10240\bigg(x_1-c_1(t)-\frac{1}{10}\bigg)\bigg(x_2-c_2(t)-\frac{1}{10}\bigg)\\
&\qquad\qquad\quad\times \bigg(x_1-c_1(t)+\frac{1}{10}\bigg)\bigg(x_2-c_2(t)+\frac{1}{10}\bigg),\\
&\widehat{u}_2(\bsx,t):=10240\bigg(x_1+c_1(t)-\frac{11}{10}\bigg)\bigg(x_2+c_2(t)-\frac{11}{10}\bigg)\\
&\qquad\qquad\quad\times \bigg(x_1+c_1(t)-\frac{9}{10}\bigg)\bigg(x_2+c_2(t)-\frac{9}{10}\bigg),\\
&c_1(t):=\frac12 + \frac14 (1-t^{10})\cos(4\pi t^2)\quad\text{and}\quad c_2(t):=\frac12
+ \frac14 (1-t^{10})\sin(4\pi t^2).
\end{align*}
As the parameters appearing in the objective functional~\eqref{eq:objective} and adjoint equation~\eqref{eq:coupled2}, we use $\alpha_1=10^{-3}$, $\alpha_2=10^{-2}$, and $\alpha_3=10^{-7}$. Moreover, we always use $$u_0(\bsx)=\sin(2\pi x_1)\sin(2\pi x_2).$$

In Subsections~\ref{sec:num_dimtrunc} and~\ref{sec:num_qmc}, we fix the source term $$z(\bsx,t)=10x_1(1-x_1)x_2(1-x_2)$$
to assess the dimension truncation and QMC errors.

All computations were carried out on the computational cluster Katana supported by Research Technology Services at UNSW Sydney~\cite{katana}.
\subsection{Dimension truncation error}\label{sec:num_dimtrunc}
The dimension truncation errors in Theorem~\ref{thm:trunc} are estimated by approximating the quantities
$$
\bigg\|\int_{U_{s'}}(u_{s'}^\bsy-u_s^\bsy)\,{\rm d}\bsy\bigg\|_{L^2(V;I)}\quad\text{and}\quad\bigg\|\int_{U_{s'}}(q_{s'}^\bsy-q_s^\bsy)\,{\rm d}\bsy\bigg\|_{L^2(V;I)}
$$
as well as
$$
\|S_{s'}-S_s\|_{L^2(V;I)}\quad\text{and}\quad |T_{s'}-T_s|
$$
for $s'\gg s$, by using a tailored lattice cubature rule generated using the fast CBC algorithm with $n=2^{15}$ nodes and a single fixed random shift to compute the high-dimensional parametric integrals. The obtained results are displayed in Figures~\ref{fig:dimtrunc0} and~\ref{fig:dimtrunc} for the fluctuations $(\psi_j)_{j\geq 1}$ corresponding to decay rates $\vartheta\in\{1.3,2.6\}$ and dimensions $s\in\{2^k\mid k\in\{1,\ldots,9\}\}$. We use $\theta=10$ in the computations corresponding to $S_s$ and $T_s$. As the reference solution, we use the solutions corresponding to dimension $s'=2048=2^{11}$.

\begin{figure}[!t]
\includegraphics[height=.34\textwidth]{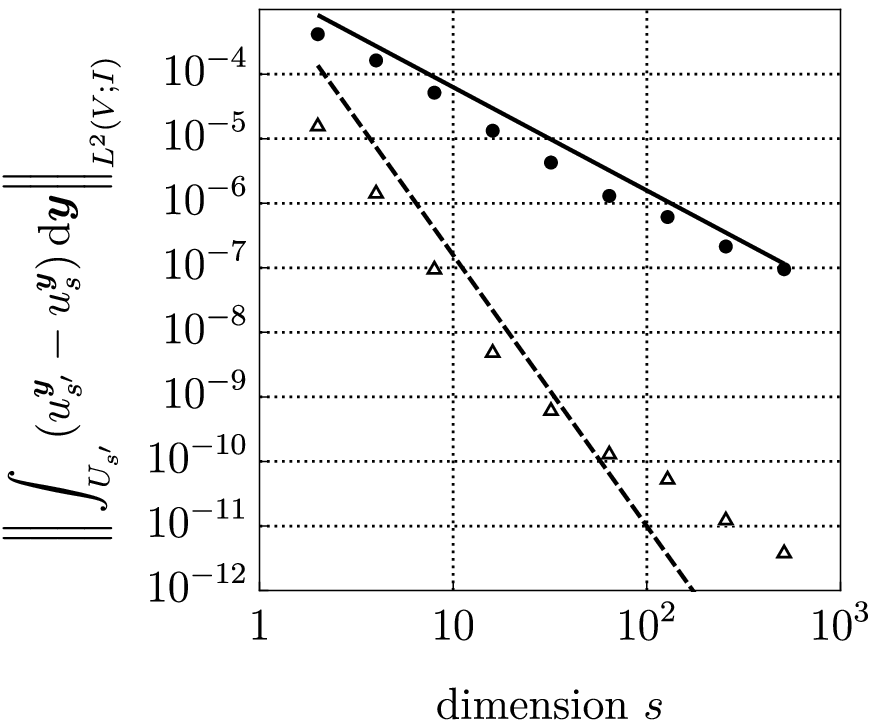}\includegraphics[trim= 0cm 0cm 3.5cm 0cm,clip,height=.34\textwidth]{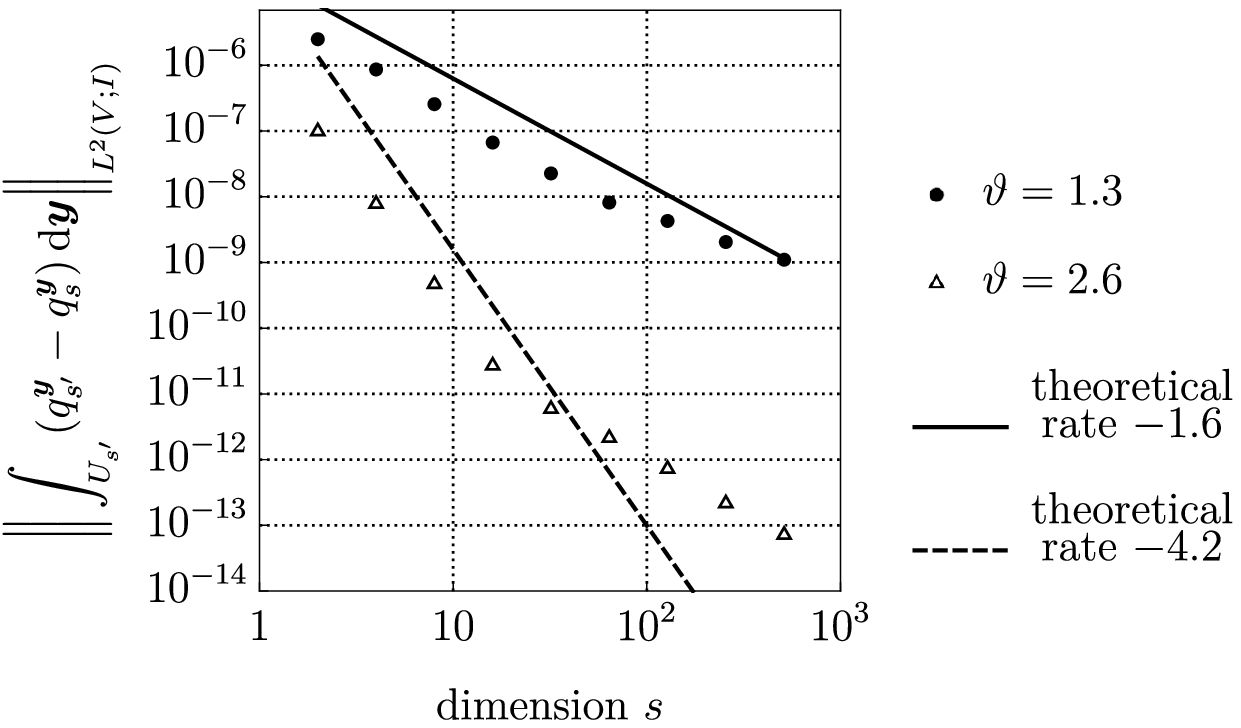}\includegraphics[trim=9.225cm 0cm 0cm 0cm,clip,height=.35\textwidth]{fig/dimtrunc_adjoint.eps}
\caption{The approximate dimension truncation errors corresponding to the state and adjoint PDEs.}\label{fig:dimtrunc0}
\end{figure}

\begin{figure}[!t]
\hspace*{.48cm}\includegraphics[height=.34\textwidth]{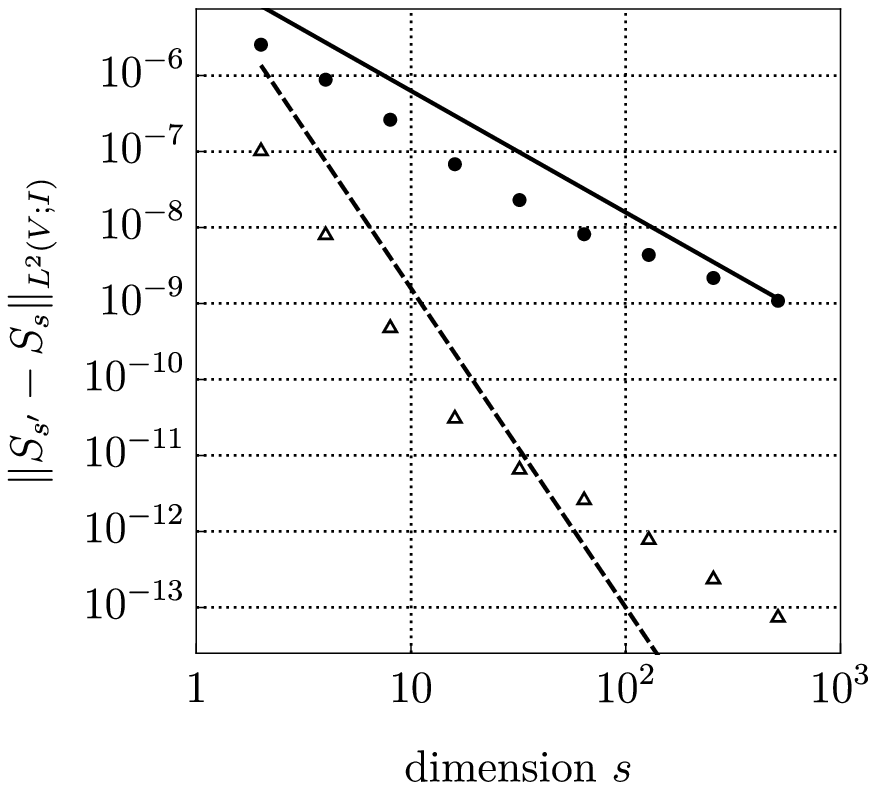}\hspace*{.43cm}\includegraphics[trim= 0cm 0cm 3.5cm 0cm,clip,height=.34\textwidth]{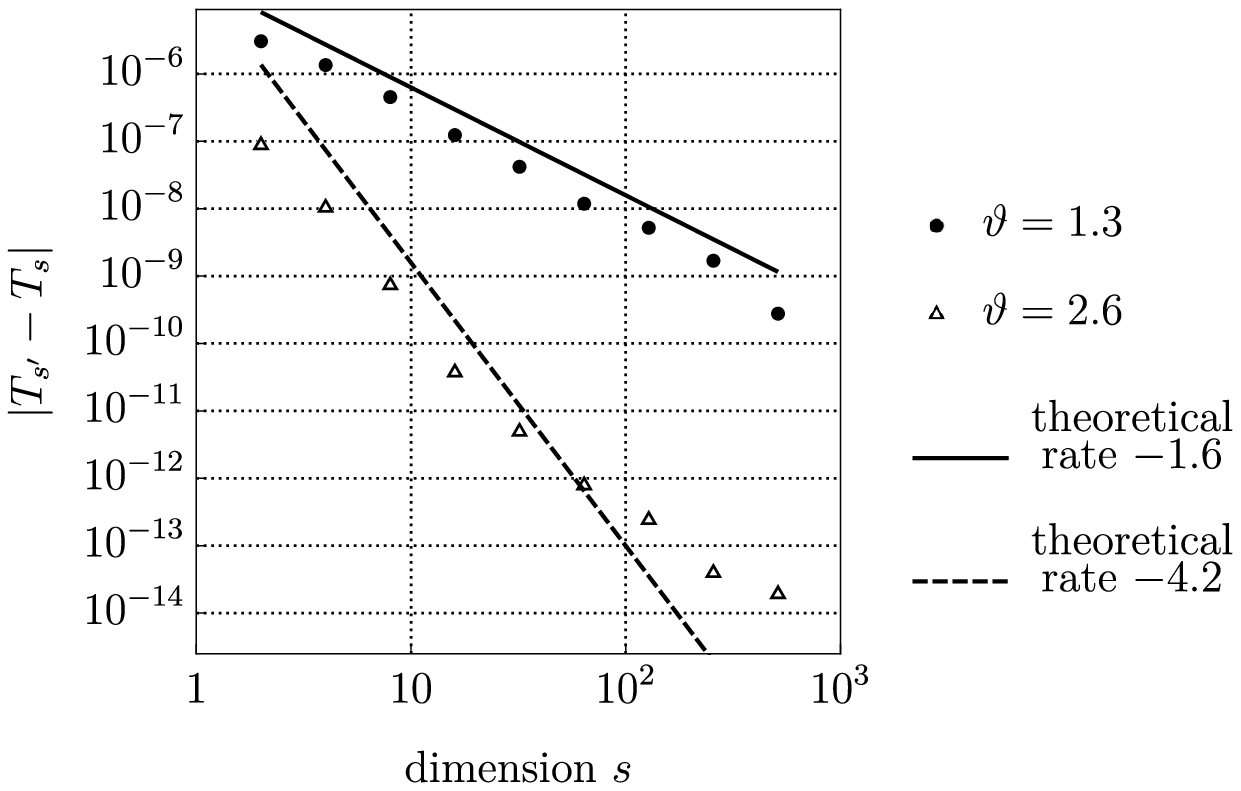}\includegraphics[trim=9.225cm 0cm 0cm 0cm,clip,height=.35\textwidth]{fig/dimtrunc_adjoint.eps}
\caption{The approximate dimension truncation errors corresponding to $\|S_{s'}-S_s\|_{L^2(V;I)}$ and $|T_{s'}-T_s|$.}\label{fig:dimtrunc}
\end{figure}

The theoretical dimension truncation rate is readily observed in the case $\vartheta=1.3$. We note in the case $\vartheta=2.6$ that the dimension truncation convergence rates degenerate for large values of $s$. This may be explained by the fact that the QMC cubature with $n=2^{15}$ nodes has an error around $10^{-8}$ (see Figure~\ref{fig:qmcerr} in Subsection~\ref{sec:num_qmc}), but the finite element discretization error may also be a contributing factor. For smaller values of $s$, the higher order convergence is also apparent in the case $\vartheta=2.6$.

\subsection{QMC error}\label{sec:num_qmc}

We investigate the QMC error rate by computing the root-mean-square approximations
\begin{align*}
&\sqrt{\frac{1}{R(R-1)}\sum_{r=1}^R \|(\overline{Q}-Q^{(r)})(u_s^{\cdot})\|_{L^2(V;I)}^2},\\
&\sqrt{\frac{1}{R(R-1)}\sum_{r=1}^R \|(\overline{Q}-Q^{(r)})(q_s^{\cdot})\|_{L^2(V;I)}^2},\\
&\sqrt{\frac{1}{R(R-1)}\sum_{r=1}^R \|(\overline{Q}-Q^{(r)})(\exp(\Phi_s^{\cdot})\,q_s^{\cdot})\|_{L^2(V;I)}^2},\\
&\sqrt{\frac{1}{R(R-1)}\sum_{r=1}^R |(\overline{Q}-Q^{(r)})(\exp(\Phi_s^{\cdot}))|^2},
\end{align*}
corresponding to~\eqref{eq:qmcerr1}--\eqref{eq:qmcerr4}, where $\overline{Q}$ and $Q^{(r)}$ are as in~\eqref{eq:qmc2} for a randomly shifted lattice rule with cubature nodes~\eqref{eq:qmc1}, where the random shift $\boldsymbol \Delta$ is drawn from $U([0,1]^s)$. As the generating vector, we use lattice rules constructed using the fast CBC algorithm with $n=2^m$, $m\in\{4,\ldots,15\}$, lattice points and $R=16$ random shifts, and $s=100$. We carry out the experiments using two different decay rates $\vartheta\in\{1.3,2.6\}$ for the input random field. The results are displayed in Figure~\ref{fig:qmcerr}. The root-mean-square error converges at a linear rate in all experiments, which is consistent with the theory.

\begin{figure}[H]
\begin{center}
\includegraphics[trim= 0cm 0cm 3.4cm 0cm,clip,height=.34\textwidth]{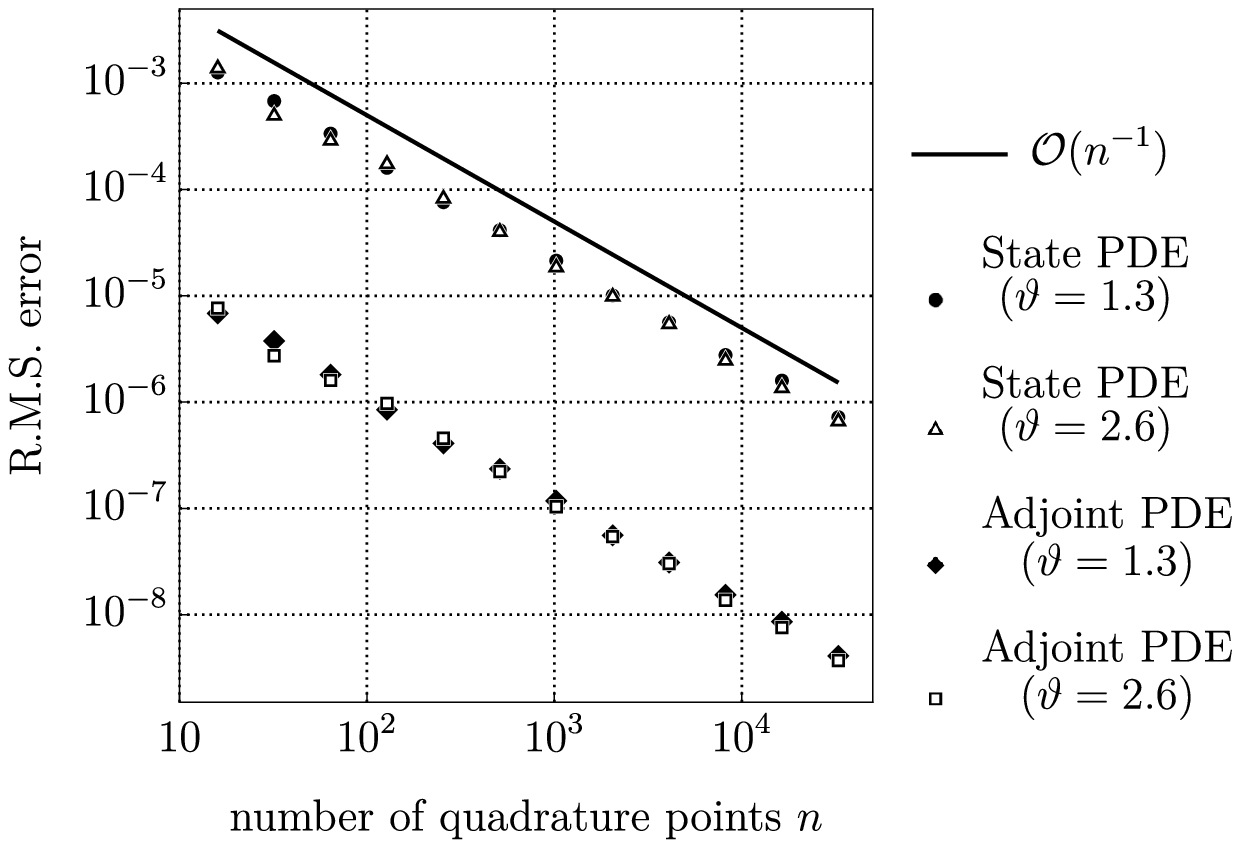}
\!\!\!\!\includegraphics[trim= 9cm 0cm 0cm 0cm,clip,height=.36\textwidth]{fig/qmcerr_stateadjoint.eps}
\!\includegraphics[trim= 0cm 0cm 3.2cm 0cm,clip,height=.34\textwidth]{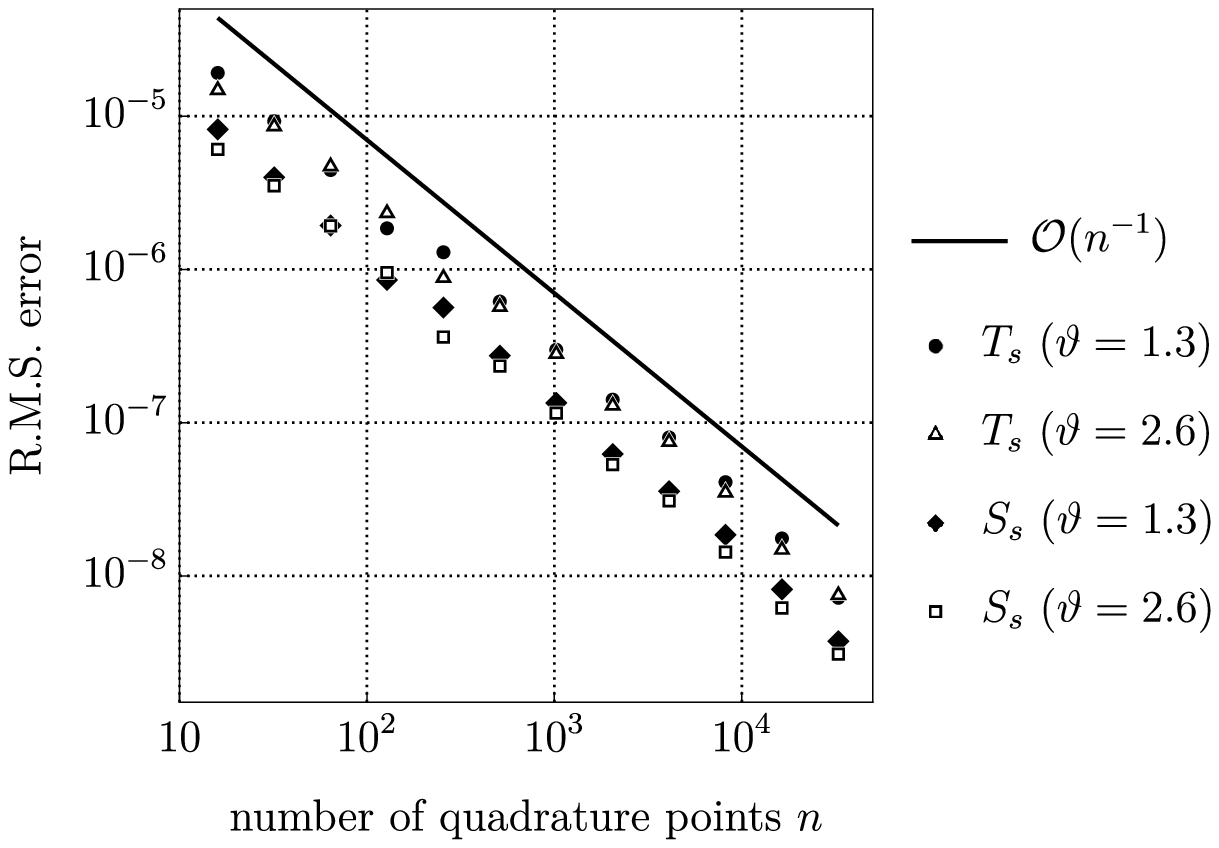}
\!\!\!\!\includegraphics[trim= 9cm 0cm 0cm 0cm,clip,height=.36\textwidth]{fig/qmcerr_entropic.eps}
\caption{Left: The approximate root-mean-square error for QMC approximation of the integrals $\int_{U_s}u_s^\bsy\,{\rm d}\bsy$ and $\int_{U_s}q_s^\bsy\,{\rm d}\bsy$. Right: The approximate root-mean-square error for QMC approximation of quantities $S_s$ and $T_s$. All computations were carried out using $R=16$ random shifts, $n=2^m$, $m\in\{4,\ldots,15\}$, and dimension $s=100$.}\label{fig:qmcerr}
\end{center}
\end{figure}

\subsection{Optimal control problem}\label{sec:num_oc}

We consider the problem of finding the optimal control $z\in \mathcal Z$
that minimizes~\eqref{eq:objective} subject to the PDE
constraint~\eqref{eq:model}. We consider constrained optimization over $\mathcal Z=\{z\in L^2(V';I):\|z\|_{L^2(V';I)}\leq 2\}$ and compare our results with the reconstruction obtained by carrying out unconstrained optimization over $\mathcal Z=L^2(V';I)$. To this end, we define the projection operator
$$
\mathcal{P}(w):=\min\bigg\{1,\frac{2}{\|w\|_{L^2(V;I)}}\bigg\}w\quad\text{for}~w\in L^2(V;I)
$$
which is used in the constrained setting, while in the unconstrained setting we use $\mathcal{P}:=\mathcal I_{L^2(V;I)}$. The operator $\mathcal{P}$ acts on $L^2(V;I)$ and hence it is different from the operator $P_{\mathcal{Z}}$ introduced in Section~\ref{sec:opt-cond}, which projects onto $\mathcal{Z}$.

To be able to handle elements of $\mathcal Z$ numerically, we apply the projected gradient method (see, e.g.,~\cite{hinze2009optimization}) as described in Algorithm~\ref{alg:projgd} together with the projected Armijo rule stated in Algorithm~\ref{alg:armijo}. Note that evaluating $J(R_Vw)$ and $J'(R_Vw)$ in Algorithms~\ref{alg:projgd} and~\ref{alg:armijo} requires solving the state PDE with the source term $R_Vw$. In particular, the Riesz operator appears in the loading term after finite element discretization and can thus be evaluated using~\eqref{eq:RV}. We use the initial guess $w_0=0$. The parameters of the gradient descent method were chosen to be $\eta_0=100$, $\gamma=10^{-4}$, and $\beta=0.1$.

We consider the entropic risk measure with $\theta=10$ and set $\vartheta = 1.3$. The reconstructed optimal control obtained using the bounded set of feasible controls $\mathcal Z$ is displayed in Figure~\ref{fig:oc}. The reconstructed optimal control at the terminal time $T=1$ and its pointwise difference to the control obtained without imposing control constraints are displayed in Figure~\ref{fig:oc2}. Finally, the evolution of the objective functional as the number of gradient descent iterations increases is plotted in Figure~\ref{fig:oc3} for the constrained and unconstrained optimization problems.

\begin{algorithm}[!t]\normalsize
\caption{Projected gradient descent}
\label{alg:projgd}
Input: feasible starting value $w\in L^2(V;I)$ such that $z=R_Vw \in \mathcal Z$

\begin{algorithmic}[1]
\WHILE{$\| w - \mathcal{P}(w - J'(R_Vw)) \|_{L^2(V;I)} >$TOL}
\STATE{find step size $\eta$ using Algorithm~\ref{alg:armijo}}
\STATE{set $w := \mathcal{P}(w - \eta J'(R_Vw))$}
\ENDWHILE
\end{algorithmic}
\end{algorithm}
\begin{algorithm}[H]
\caption{Projected Armijo rule}
\label{alg:armijo}\normalsize
Input: current $w\in L^2(V;I)$, parameters $\beta,\gamma \in (0,1)$ and $\eta_0>0$\\
Output: step size $\eta > 0$
\begin{algorithmic}[1]
\STATE{set $\eta := \eta_0$}
\WHILE{\\$J(R_V \mathcal{P}(w - \eta J'(R_Vw)))- J(R_Vw) > -\frac{\gamma}{\eta} \|w-\mathcal{P}(w - \eta J'(R_Vw))\|^2_{L^2(V;I)}$}
\STATE{set $\eta := \beta \eta$}
\ENDWHILE
\end{algorithmic}
\end{algorithm}

\begin{figure}[H]
\begin{center}
\includegraphics[trim= .7cm 0cm 5.3cm 0cm,clip,width=.992\textwidth]{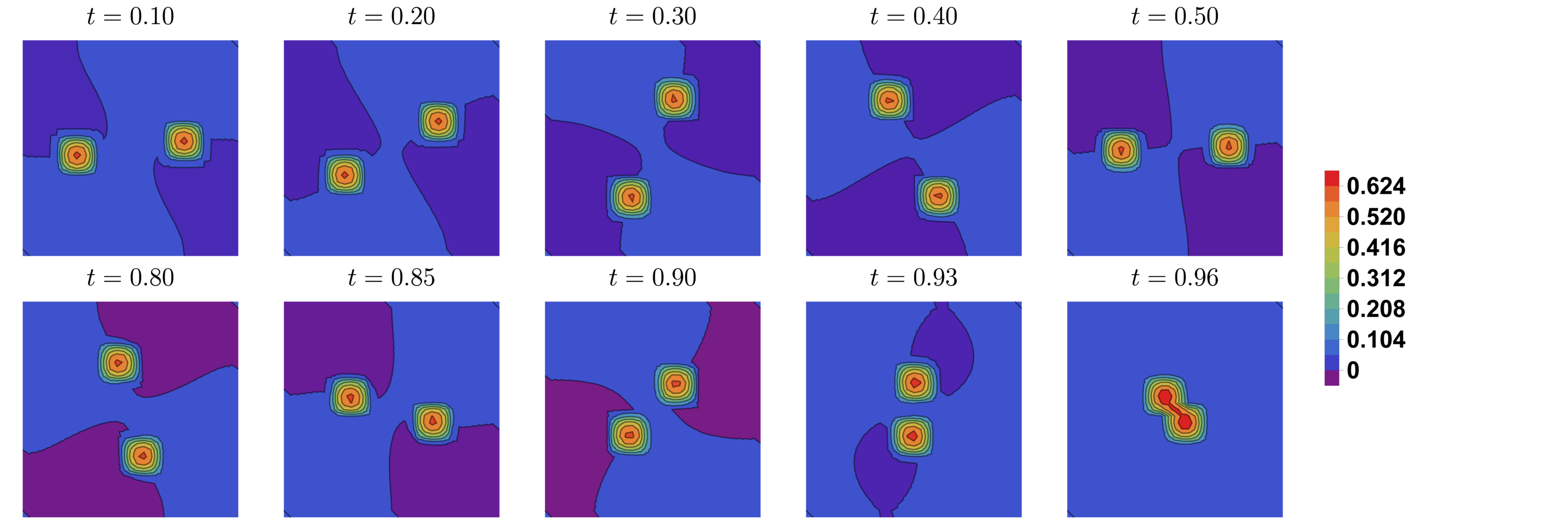}
\caption{The inverse Riesz transform $R_V^{-1}z^*$ of the reconstructed optimal control $z^*$ using the entropic risk measure for several values of $t$ in the constrained setting.}\label{fig:oc}
\end{center}
\end{figure}

\begin{figure}[H]
\begin{center}
\includegraphics[trim= 0cm 3cm 0cm 3cm,clip,height=.255\textwidth]{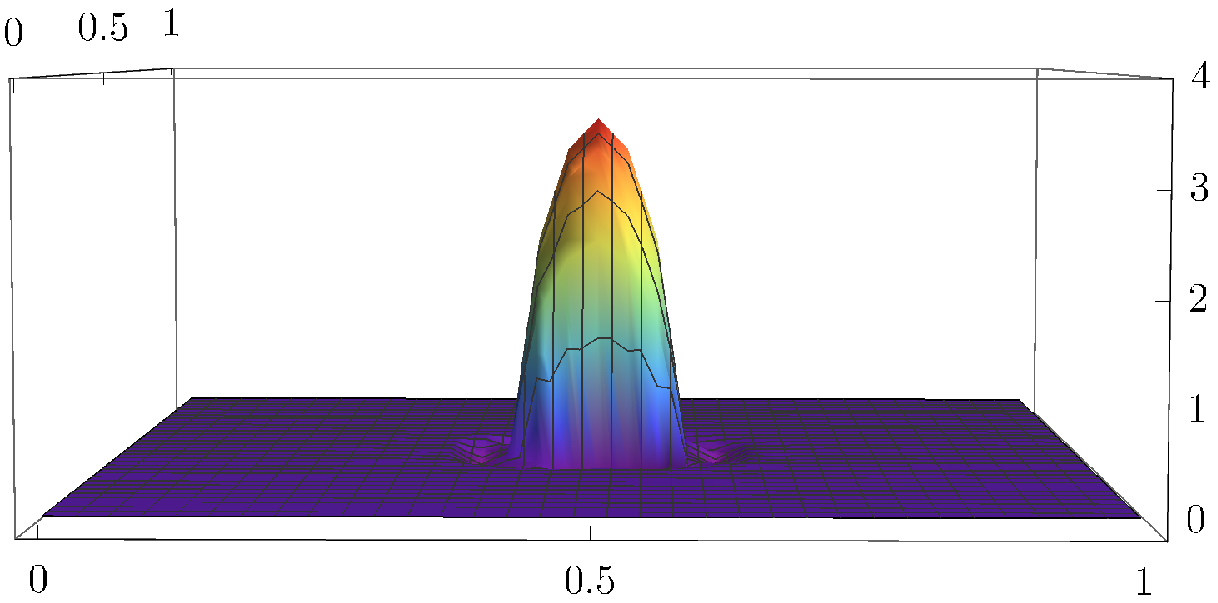}\includegraphics[trim= 0cm 3cm 0cm 3cm,clip,height=.255\textwidth]{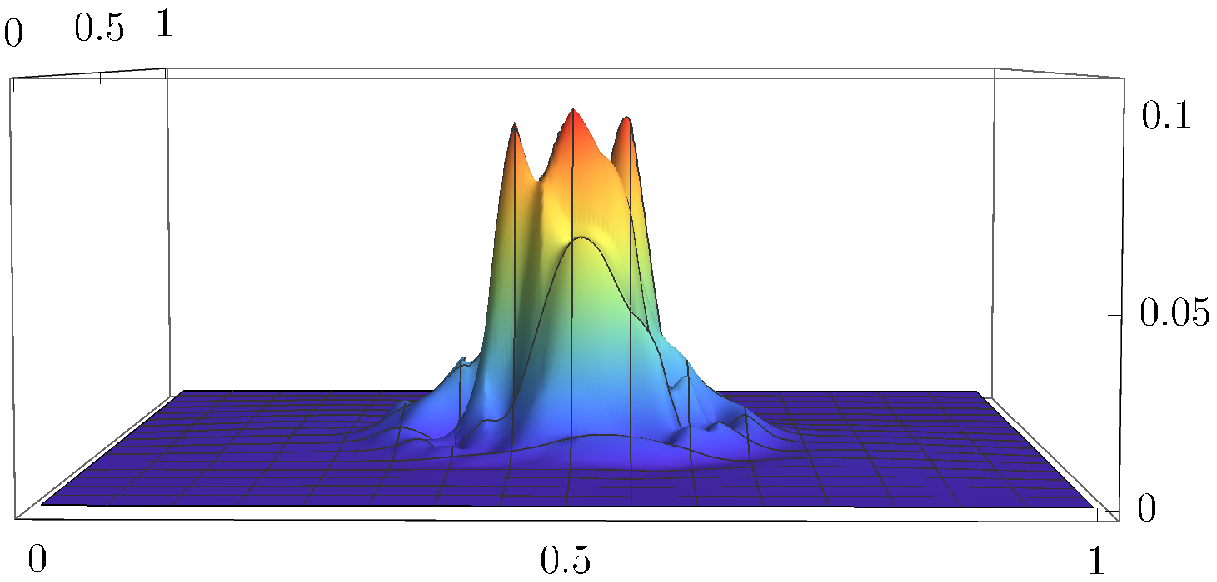}
\caption{Left: the inverse Riesz transform of the control at time $t=1$ in the constrained setting after 25 iterations of the projected gradient descent algorithm using the entropic risk measure. Right: The difference between the reconstruction obtained in the constrained setting and the corresponding solution in the unconstrained setting.}\label{fig:oc2}
\end{center}
\end{figure}

\begin{figure}[H]
\begin{center}
\includegraphics[width=.93\textwidth]{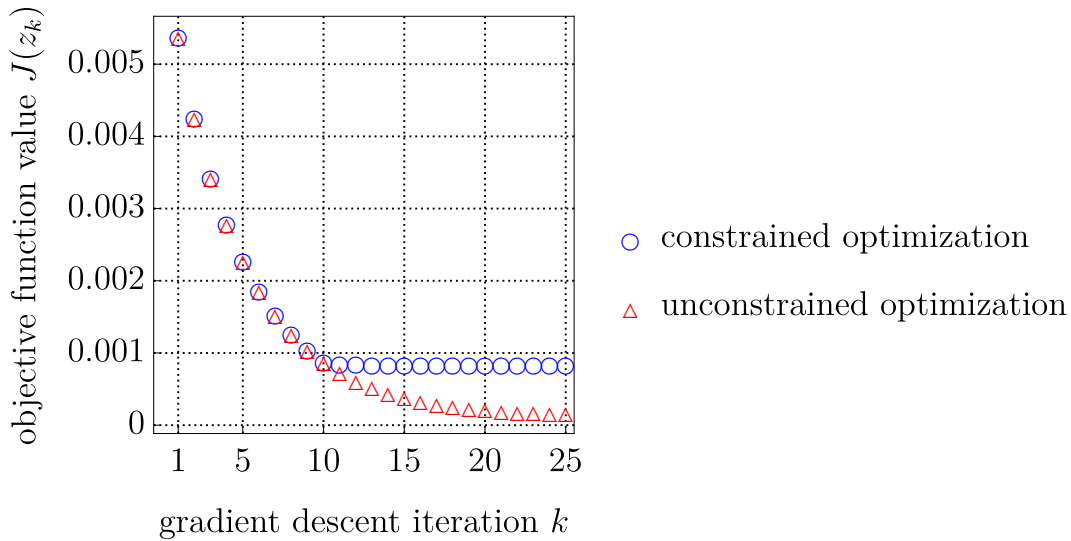}
\caption{The value of the objective functional for each gradient descent iteration. The results corresponding to the constrained setting and the unconstrained setting are plotted in blue and red, respectively.}\label{fig:oc3}
\end{center}
\end{figure}

\section{Conclusion} \label{sec:conclusion}

We developed a specially designed QMC method for an optimal control problem subject to a parabolic PDE with an uncertain diffusion coefficient.  To account for the uncertainty, we considered as measures of risk the expected value and the more conservative (nonlinear) entropic risk measure.  For the high-dimensional integrals originating from the risk measures, we provide error bounds and convergence rates in terms of dimension truncation and the QMC approximation. In particular, after dimension truncation, the QMC error bounds do not depend on the number of uncertain variables, while leading to faster convergence rates compared to Monte Carlo methods.  In addition we extended QMC error bounds in the literature to separable Banach spaces,  and hence the presented error analysis is discretization invariant.

\section*{Acknowledgements}
P.~A. Guth is grateful to the DFG RTG1953 ``Statistical Modeling of Complex Systems and Processes'' for funding of this research. F.~Y. Kuo and I.~H. Sloan acknowledge the support from the Australian Research Council (DP210100831).

\bibliographystyle{plain}
\bibliography{OCuU}

\end{document}